\documentclass{amsart}[11pt,a4paper]
\usepackage{amscd,amsfonts,amssymb,amsmath,enumerate,mathtools,pinlabel,float,lscape}
\usepackage{blindtext}
\usepackage{tikz}
\usepackage{tikz-cd}
\usepackage[myheadings]{fullpage}
\usepackage{hyperref}
\hypersetup{
    colorlinks=true,   
    linkcolor=blue,
    citecolor=blue,
}

\usepackage[all,cmtip]{xy}

\newtheorem{theorem}{Theorem}[section]
\newtheorem{proposition}[theorem]{Proposition}
\newtheorem{lemma}[theorem]{Lemma}
\newtheorem{corollary}[theorem]{Corollary}

\theoremstyle{definition}
\newtheorem{definition}[theorem]{Definition}

\newtheorem{remark}[theorem]{Remark}

\theoremstyle{plain}
\numberwithin{equation}{section}
\newtheorem{theorem*}{Theorem}
\newtheorem{proposition*}[theorem*]{Proposition}
\newtheorem{corollary*}[theorem*]{Corollary}

\def \Z{\mathbb Z}

\newcommand{\Mod}{\mathrm{Mod}}
\newcommand{\PMod}{\mathrm{PMod}}
\newcommand{\UMod}{\mathrm{UMod}}
\newcommand{\Homeo}{\mathrm{Homeo}}
\newcommand{\LMod}{\mathrm{LMod}}
\newcommand{\SMod}{\mathrm{SMod}}
\renewcommand{\S}{\mathbb{S}}
\renewcommand{\L}{\mathcal{L}}
\newcommand{\I}{\mathcal{I}}

\newcommand{\C}{\mathcal{C}}

\newcommand{\Sp}{\mathrm{Sp}}
\newcommand{\SL}{\mathrm{SL}}

\newcommand{\Stab}{\mathrm{Stab}}
\numberwithin{equation}{subsection}

\makeatletter
\@namedef{subjclassname@2020}{\textup{2020} Mathematics Subject Classification}
\makeatother

\begin{document}
	\title[Generating $\LMod$ of regular cyclic covers]{Generating the liftable mapping class groups \\of regular cyclic covers}

	%    Information for author 1
	\author[S. Dey]{Soumya Dey}
	\address{Krea University, 5655, Central Expressway, Sri City, Andhra Pradesh 517646 India}
	\email{soumya.sxccal@gmail.com}
	\urladdr{https://sites.google.com/site/soumyadeymathematics}
	
	%    Information for author 2
	\author[N. K. Dhanwani]{Neeraj K. Dhanwani}
	\address{Department of Mathematical Sciences\\
		Indian Institute of Science Education and Research Mohali\\
		Knowledge city, Sector 81, SAS Nagar, Manauli PO 140306}
	\email{neerajk.dhanwani@gmail.com}
	
	%    Information for author 3
	\author[H. Patil]{Harsh Patil}
	\address{Department of Mathematics\\
		University of Bristol\\
		Beacon House, Queens Road \\
		Bristol, BS8 1QU, UK}
	\email{cr22307@bristol.ac.uk}
	
	%    Information for author 4
	\author[K. Rajeevsarathy]{Kashyap Rajeevsarathy}
	\address{Department of Mathematics\\
		Indian Institute of Science Education and Research Bhopal\\
		Bhopal Bypass Road, Bhauri \\
		Bhopal 462 066, Madhya Pradesh\\
		India}
	\email{kashyap@iiserb.ac.in}
	
	\urladdr{https://home.iiserb.ac.in/$_{\widetilde{\phantom{n}}}$kashyap/}
	
	
	\begin{abstract}
Let $\mathrm{Mod}(S_g)$ be the mapping class group of the closed orientable surface of genus $g \geq 1$, and let $\mathrm{LMod}_{p}(X)$ be the liftable mapping class group associated with a finite-sheeted branched cover $p:S \to X$, where $X$ is a hyperbolic surface. For $k \geq 2$, let $p_k: S_{k(g-1)+1} \to S_g$ be the standard $k$-sheeted regular cyclic cover. In this paper, we show that $\{\mathrm{LMod}_{p_k}(S_g)\}_{k \geq 2}$ forms an infinite family of self-normalizing subgroups in $\mathrm{Mod}(S_g)$, which are also maximal when $k$ is prime. Furthermore, we derive explicit finite generating sets for $\mathrm{LMod}_{p_k}(S_g)$ for $g \geq 3$ and $k \geq 2$, and $\mathrm{LMod}_{p_2}(S_2)$. For $g \geq 2$, as an application of our main result, we  also derive a generating set for $\mathrm{LMod}_{p_2}(S_g) \cap C_{\mathrm{Mod}(S_g)}(\iota)$, where $C_{\mathrm{Mod}(S_g)}(\iota)$ is the centralizer of the hyperelliptic involution $\iota \in \mathrm{Mod}(S_g)$. Let $\L$ be the infinite ladder surface, and let $q_g : \L \to S_g$ be the standard infinite-sheeted cover induced by $\langle h^{g-1} \rangle$ where $h$ is the standard handle shift on $\mathcal{L}$. As a final application, we derive a finite generating set for $\mathrm{LMod}_{q_g}(S_g)$ for $g \geq 3$.  
\end{abstract}
	
\maketitle
\section{Introduction}
\label{sec:intro}
	
Let $S_g$ be the closed connected orientable surface of genus $g \ge 1$, and let $\Mod(S_g)$ be the mapping class group  of $S_g$. Understanding the finite-index subgroups of $\Mod(S_g)$ and their homology groups is a natural pursuit (see \cite[\textsection 2 and \textsection 6]{DM1}). A special class of finite-index subgroups of $\Mod(S_g)$ arise as the liftable mapping class groups of finite-sheeted covers. The study of the liftable mapping class groups can be traced back to the seminal works of Birman-Hilden~\cite{BH1,BH2,BH3} in the 1970s. In the past decade,  there has also been some renewed interest~\cite{GW3,GW2,GW1} on the study of these groups.  The problem of liftability of periodic mapping classes under finite cyclic branched covers is tackled combinatorially in~\cite{RS1} using the theories of orbifolds and automorphisms of compact Riemann surfaces. For $k \ge 1$ and $g_k := k(g-1)+1$,  let $p_k : S_{g_k} \to S_g$ be the standard regular $k$-sheeted cover of $S_g$ induced by a $\Z_k$-action on $S_{g_k}$, as shown in Figure \ref{fig:regular_cover} below (for $k = 4$).  Let $\LMod_{p}(X)$ denote the liftable mapping class group of a finite-sheeted branched cover $p:S \to X$. Recently, in~\cite{ADDR21}, a symplectic characterization of elements in $\LMod_{p_k}(S_g)$ was derived. Furthermore, it was shown that for $g \geq 2$, $\LMod_{p_k}(S_g)$ is a generalization of the congruence subgroup $\Gamma_0(k) (= \LMod_{p_k}(S_1))$ of $\SL(2,\Z)$ which appears frequently in the theory of modular forms~\cite{HA,IO,TM,RS}. In this paper, we show that for $k \geq 2$, $\LMod_{p_k}(S_g)$ is a self-normalizing subgroup of $\Mod(S_g)$ which is maximal when $k$ is prime. Furthermore, we derive an explicit finite generating set for $\LMod_{p_k}(S_g)$ for $k \geq 2$ and $g \geq 3$, and also for the case when $(g,k) = (2,2)$.

	\begin{figure}[htbp]
		\labellist
		\small
		\pinlabel $a_1$ at 49 22
		\pinlabel $b_1$ at 70 80
		
		\pinlabel $c_1$ at 100 32
		
		\pinlabel $a_2$ at 118 22
		\pinlabel $b_2$ at 140 80
		
		\pinlabel $a_g$ at 320 22
		\pinlabel $b_g$ at 342 80
		
		\pinlabel $\pi/2$ at 410 360
		\pinlabel $p_4$ at 220 160
		
		\endlabellist
		\centering
		\includegraphics[width=40ex]{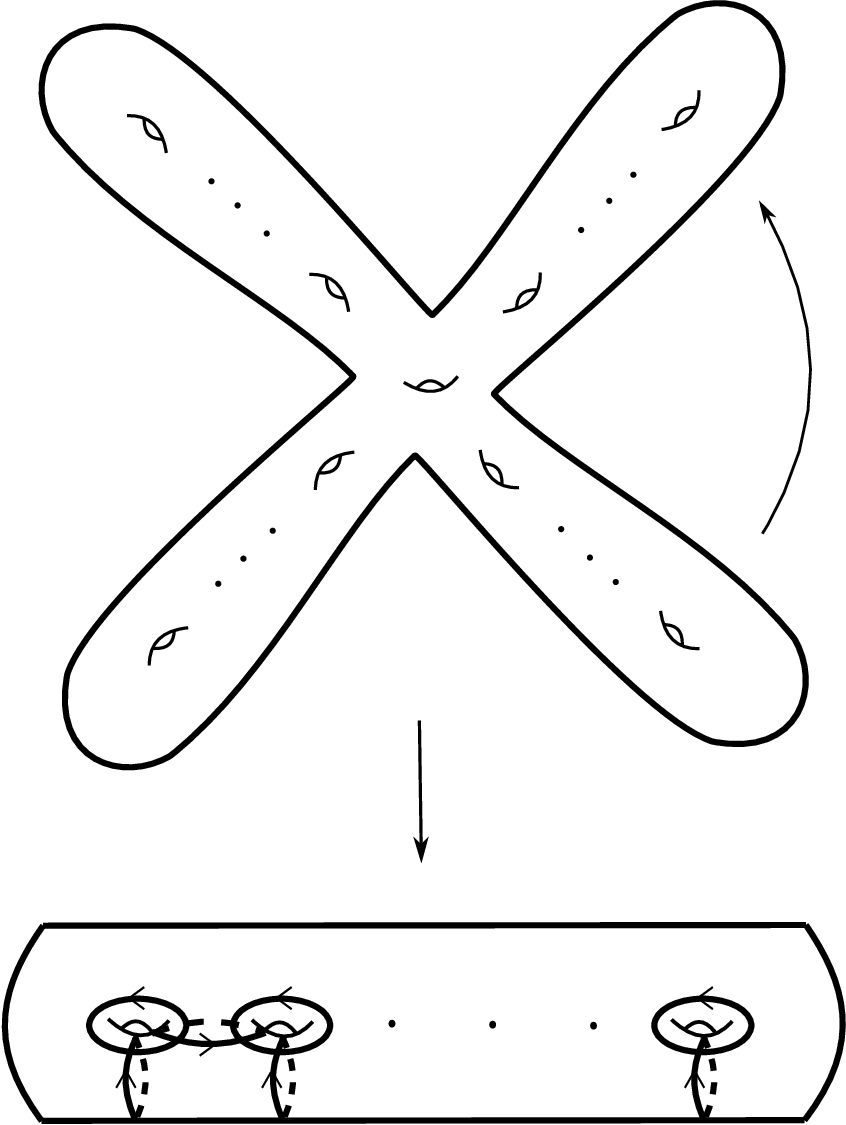}
		\caption{A free regular cover $p_4: S_{g_4} \to S_g$ induced by a $\Z_4$-action on $S_{g_{4}}.$}
		\label{fig:regular_cover}
	\end{figure}
	
For $k$ prime, it was shown in~\cite{ADDR21} that $[\Mod(S_g):\LMod_{p_k}(S_g)] = \sum_{i=0}^{2g-1} k^i$ and that $\LMod_{p_k}(S_g)$ is not a maximal subgroup of $\Mod(S_g)$ when $k$ is composite. Moreover, using the main result of \cite{BGP}, it can be shown that $\LMod_{p_2}(S_g)$ is a maximal subgroup of $\Mod(S_g)$. Thus, a natural question that arises in this context is whether $\LMod_{p_k}(S_g)$ is maximal when $k$ is prime. In Section~\ref{sec:alg_sig_lmod}, using the symplectic criterion derived in~\cite{ADDR21} we show the following. 
\begin{theorem*}
\label{alg_props_lmod}
For $g \geq 2$ and $k \geq 2$, we have:
\begin{enumerate}[(i)]
\item $\LMod_{p_k}(S_g)$ is self-normalizing in $\Mod(S_g)$, and
\item when $k$ is prime, $\LMod_{p_k}(S_g)$ is maximal in $\Mod(S_g)$. 
\end{enumerate}
\end{theorem*}
	
Let $\Psi: \Mod(S_g) \to \Sp(2g, \Z)$ be the symplectic representation induced by the action of $\Mod(S_g)$ on $H_1(S_g, \Z)$, that preserves the algebraic intersection number $\hat{i}(\alpha,\beta)$ between any two classes $\alpha, \beta \in H_1(S_g, \mathbb{Z})$. Let $\Psi_k$ denote the composition of the map $\Psi$ with the modulo $k$ reduction map $\Sp(2g, \Z) \to \Sp(2g, \Z_k)$, and let $\Mod(S_g)[k] := \ker(\Psi_k)$ (also known as the \textit{level-$k$ subgroup} of $\Mod(S_g)$). Let $\{a_1,b_1, \ldots,a_g,b_g \}$ be the standard generators of $H_1(S_g,\Z) (\cong \Z^{2g})$ represented by the oriented curves in Figure~\ref{fig:regular_cover} and the standard unit vectors $e_1, \ldots, e_{2g}$ in $\Z^{2g}$, respectively.

Note that $p_{k}$ induces an injective map $(p_{k})_*:\pi_1(S_{g_k}) \to \pi_1(S_g)$. In fact $(p_k)_*(\pi_1(S_{g_k}))$ is a normal subgroup of $\pi_1(S_g)$, which is the kernel of the map $\phi:\pi_1(S_g)\to \mathbb{Z}_k$ defined by $c\xmapsto{\phi} \hat{i}(a_1,c) \pmod k$, where $c$ on the right refers to the class it represents in $H_1(S_g, \mathbb{Z})$. Let $(p_k)_{\#}: H_1(S_{g_k},\Z) \to H_1(S_g,\Z)$ be the induced map on homology. By identifying $\pi_1(S_{g_k})$ with  $\ker \, \phi < \pi_1(S_g)$, we obtain $(p_k)_{\#} (H_1(S_{g_k},\Z))  = \langle a_1,b_1^k, a_2,b_2,\ldots,a_g,b_g\rangle$. Let $\Mod_{p_k}(S_g,e_1) := \Stab_{\Mod(S_g)}(e_1)$, where $e_1 = (1, 0,\ldots,0) \in H_1(S_g,\Z_k)$.

 In \cite{ADDR21}, it was shown that for $g \geq 2$, there exists a normal series in $\Sp(2g,\Z)$ given by
\[\tag{$\ast$} 1 \lhd \Psi(\Mod(S_g)[k]) \lhd \Psi(\Mod_{p_k}(S_g, e_1)) \lhd \Psi(\LMod_{p_k}(S_g))\]
that generalizes a well known~\cite[\textsection 4.2]{TM} normal series of congruence subgroups in $\SL(2,\Z)$.
Moreover, it was shown that the normal series (*) pulls back under $\Psi$ to an analogous series in $\Mod(S_g)$ with $[\LMod_{p_k}(S_g):\Mod_{p_k}(S_g,e_1)] = \phi(k)$, where $\phi$ is the Euler's totient function. Furthermore, a set of coset representatives of $\LMod_{p_k}(S_g)/\Mod_{p_k}(S_g, e_1)$ are given by 
$$S_k'' := \{T_{b_1}^{\bar{\ell}-1} T_{a_1}^{-1} T_{b_1}^{\ell -1} : \ell \in \Z_k^{\times} \text{ and } \ell \bar{\ell} \equiv 1 \pmod{k}\}.$$
\noindent Thus, our approach in this paper involves the derivation of a finite generating set for $\Mod_{p_k}(S_g,e_1)$. Again, from~\cite{ADDR21}, we know that $\Mod_{p_k}(S_1,e_1)$ is the congruence subgroup $\Gamma_1(k) < \SL(2,\Z)$, where
$$\Gamma_1(k) = \left\{ \begin{bmatrix} a & b \\ c & d \end{bmatrix} \in \SL(2,\Z) : a, d \equiv 1 \pmod{k}, \,c \equiv 0 \pmod{k}\right \}.$$
Let $T_c$ denote the left-handed Dehn twist about a simple closed curve $c$ in $S_g$. Given any finite generating set $\S_k$ for $\Gamma_1(k)$ (see \cite[Lemma 12]{DKMO} and \cite[Proposition 1.17]{WS})  expressed as a product of the matrices $\Psi(T_{a_1})$ and $\Psi(T_{b_1})$, consider a collection $\tilde{\S}_k$ of words (mapping classes) in $\Mod(S_1)$ in $\{T_{a_1},T_{b_1}\}$ such that $\Psi(\tilde{\S}_k) =\S_k$ and $|\S_k| = |\tilde{\S}_k|$. Our main result can be stated as follows.

\begin{theorem*}\label{theorem1}
For $k \geq 2$ and $g \geq 3$, 
$$\Mod_{p_k}(S_g,e_1) = \langle \ \tilde{\S}_k \cup \S_{g,k} \cup \{T_{a_2},T_{b_2}, \ldots,T_{a_g},T_{b_g},T_{c_1},\ldots,T_{c_{g-1}}\}\rangle,$$
where $\S_{g,k} = \{F_1,\ldots,F_{g-2}\}$ is a collection of bounding pair maps in the Torelli group $\I(S_g)$ such that for each $i$, $F_i$ corresponds to a bounding pair of genus $i$ (see Figure~\ref{fig:torelli_gens} below) that intersects the curve $a_1$ nontrivially. Consequently, 
$$\LMod_{p_k}(S_g) = \langle \ \S_k'' \cup \tilde{\S}_k \cup \S_{g,k} \cup \{T_{a_2},T_{b_2}, \ldots,T_{a_g},T_{b_g},T_{c_1},\ldots,T_{c_{g-1}}\}\rangle.
$$
\end{theorem*} 
\begin{figure}[htbp]
		\includegraphics[width=50ex]{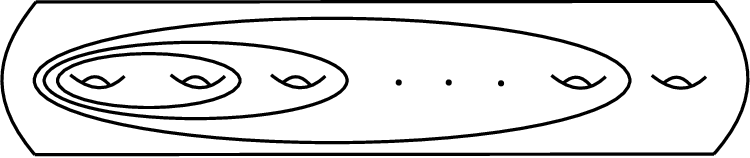}
		\caption{Each curve in the figure along with its image under the hyperelliptic involution $\iota$ forms a bounding pair that determines a map of type $F_i$.}
		\label{fig:torelli_gens}
	\end{figure}

\noindent The component groups of ($\ast$) (which are the congruence subgroups of the paramodular group $\Sp(2g,\Z)$) are of great importance in the theory of Siegel modular forms~\cite{HA,IO,RS}. Thus, it may be noted that the proof of our main result (see \textsection \ref{sec:gen_lmod}) also yields generating sets for the groups $\Psi(\LMod_{p_k}(S_g))$ and $\Psi(\Mod_{p_k}(S_g,e_1))$, for $g,k \geq 2$. Let $R_{g,k} \in \Mod(S_{g_k})$ be represented by the free $2\pi/k$ rotation of $S_{g_k}$ that induces the cover $p_k$. As an immediate application of Theorem~\ref{theorem1} we derive an explicit generating set for the centralizer of $R_{g,2}$ in $\Mod(S_{2g-1})$ for $g \geq 3$ (see Corollary~\ref{thm:cent_gen_set}). 

In the past decade, there has been a growing interest around mapping class groups of surfaces of infinite type, also known as \textit{big mapping class groups} (see \cite{AV20} and the references therein). Since infinite-type surfaces are inverse limits of finite-type surfaces, big mapping class groups are used to analyze the asymptotic and stable properties of mapping class groups of finite-type surfaces. In Section \ref{sec:ladder}, we consider the covering map $q_g : \mathcal{L} \to S_g$ induced by $\langle h^{g-1} \rangle$ where $h$ is the standard handle shift on the infinite ladder surface $\mathcal{L}$ (see Figure~\ref{fig:ladder}). We show that this cover satisfies the Birman-Hilden property~\cite{MW}, and establish the following. 
	\begin{theorem*}
For $g \geq 2$, let $\displaystyle \UMod(S_g) :=  \bigcap_{k \geq 2} \LMod_{p_k}(S_g)$ and let $\Mod(S_g,e_1)$ be the stabilizer subgroup of $\Mod(S_g)$ with respect to the vector $e_1 \in H_1(S_g,\Z)$. Then: 
\begin{enumerate}[(i)]
\item $\UMod(S_g) = \LMod_{q_g}(S_g)$, and consequently for $g \geq 3$, we have
		$$\UMod(S_g) = \langle \S_{g,k}  \cup \{\iota, T_{a_1},T_{a_2},T_{b_2}, \ldots,T_{a_g},T_{b_g},T_{c_1},\ldots,T_{c_{g-1}}\} \rangle.$$
\item There exists a normal series of $\UMod(S_g)$ in $\Mod(S_g)$ given by 
$$1 \lhd \I(S_g) \lhd \Mod(S_g,e_1) \lhd \UMod(S_g),$$ 
where the two distinct cosets of $\UMod(S_g)/\Mod(S_g,e_1)$ are represented by the elements in $\langle \iota \rangle$.
\end{enumerate}
\end{theorem*}

Let $S_{g,n}$ denote the surface of genus $g$ with $n$ marked points. Let $\delta : S_g \to S_{0,2g+2}$ be the branched cover induced by the action of $\langle \iota \rangle$ on $S_g$, and let $C_{\Mod(S_g)}(\iota)$ be the centralizer of $\{\iota\}$ in $\Mod(S_g)$. The group $C_{\Mod(S_g)}(\iota)$ which is commonly known as the \textit{hyperelliptic mapping class group} (and formerly known as the \textit{symmetric mapping class group}~\cite{BH1,BH2,BH3}) has been widely studied~\cite{FRC,NK}. The techniques used in the proof of Theorem~\ref{theorem1} do not work for the case $g = 2$ since $\I(S_2)$ is not finitely generated~\cite{MM}. In Section~\ref{sec:lmod_p2_s2}, we use the well-known~\cite{BH3} epimorphism $\hat{\delta} : C_{\Mod(S_g)}(\iota) \to \Mod(S_{0,2g+2})$ induced by $\delta$ to show the following. 

\begin{theorem*}
For $g \geq 2$, we have: 
\begin{enumerate}[(i)]
\item $\LMod_{p_2}(S_g) \cap C_{\Mod(S_g)}(\iota) = \langle \iota, T_{a_1},T_{a_g}, T_{b_1}^2, T_{b_2}. \ldots T_{b_g}, T_{c_1}, \ldots, T_{c_{g-1}}\rangle$ and
\item $\LMod_{p_2}(S_2) = \langle T_{a_1}, T_{b_1}^2, T_{c_1}, T_{a_2}, T_{b_2} \rangle$.
\end{enumerate}
\end{theorem*}

\section{Algebraic significance of $\LMod_{p_k}(S_g)$ in $\Mod(S_g)$}
	\label{sec:alg_sig_lmod}
	Let $\Psi: \Mod(S_g) \to \Sp(2g, \Z)$ be the symplectic representation induced by the action of $\Mod(S_g)$ on $H_1(S_g, \Z)$, and let $\Psi_k$ denote the composition of the map $\Psi$ with the modulo $k$ reduction map $\Sp(2g, \Z) \to \Sp(2g, \Z_k)$. Let $\I(S_g)$ denote the Torelli group, which is the kernel of $\Psi$. Restricting $\Psi$ to $\Mod_{p_k}(S_g,e_1) = \Stab_{\Mod(S_g)}(e_1)$, where $e_1 = (1, 0,\ldots,0) \in H_1(S_g,\Z_k)$, we have a short exact sequence	$$1 \to \I(S_g) \to \Mod_{p_k}(S_g,e_1) \xrightarrow{\Psi} \Psi(\Mod_{p_k}(S_g,e_1)) \to 1.$$
	
	In \cite{ADDR21}, the following symplectic criteria for liftability of mapping classes under $p_k$ have been deduced, which we will use in this section. 
\begin{theorem}[{\cite[Theorem 2.2]{ADDR21}}]
\label{thm:symp_lift_crit}
Given an $f \in \Mod(S_g)$, the following statements are equivalent.
\begin{enumerate}[(i)]
\item $f \in \LMod_{p_k}(S_g)$.
\item $\Psi(f) = (d_{ij})_{2g \times 2g}$, where $k|d_{2i}$, for $1\leq i\leq 2g$ and $i \neq 2$, and $\gcd(d_{22},k)=1$.
\item $\Psi_k(f) = (e_{ij})_{2g \times 2g}$, where $e_{2i}=0$, for $1\leq i\leq 2g$ and $i \neq 2$, and $e_{22} \in \Z_k^{\times}$.
\item $f \in \Stab_{\Mod(S_g)}(\{\ell e_1:\ell \in \Z_k^{\times}\})$.
\end{enumerate}
\end{theorem}
\noindent We will also require the following result which combines Theorem 3.1 and Corollary 3.3 from~\cite{ADDR21}. 

\begin{theorem}
\label{thm:norm_series_gen}
For $g \ge 1$ and $k \geq 2$, there exists a normal series of $\LMod_{p_k}(S_g)$ given by
$$ 1 \lhd \Mod(S_g)[k] \lhd \Mod_{p_k}(S_g, e_1) \lhd \LMod_{p_k}(S_g), \text{ where}$$ $$\LMod_{p_k}(S_g)/\Mod_{p_k}(S_g, e_1) \cong \Z_k^{\times} \text{ and }$$
$$[\Mod(S_g) : \Mod_{p_k}(S_g,e_1)] = | \{ \text{Primitive vectors in } \Z_k^{2g}\}|.$$ Moreover, the distinct cosets of $\LMod_{p_k}(S_g)/\Mod_{p_k}(S_g, e_1)$ are represented by the elements in $$\S_k'' := \{T_{b_1}^{\bar{\ell}-1} T_{a_1}^{-1} T_{b_1}^{\ell -1} : \ell \in \Z_k^{\times} \text{ and } \ell \bar{\ell} \equiv 1 \pmod{k}\}.$$
\end{theorem}

\subsection{Maximality of $\LMod_{p_k}(S_g)$}

In this subsection, we discuss the maximality of $\LMod_{p_k}(S_g)$ in $\Mod(S_g)$ when $k$ is prime. In the process, we also give a set of coset representatives of $\LMod_{p_k}(S_g)$ in $\Mod(S_g)$.  Let $G$ be a group and $H\subset G$ be a finite-index subgroup of $G$. Suppose $[G:H]=n$, and $\{x_i : 1\leq i\leq n\}$ be the left coset representatives. Let $H_i$ denote the left coset of $G/H$ given by $x_iH$. Without loss of generality, we assume that $x_1=1$ and $H_1=H$. We will require the following elementary group-theoretic lemma. 

\begin{lemma}\label{lem:imp_max}
Under the notations above, we have: 
\begin{enumerate}[(i)]
\item $\langle H, x_iH\rangle=\langle H,y\rangle$ for any $y\in x_iH$, and
\item if there exists some $y\in G$ such that $y\notin H\cup H_i$ but $y \in \langle H, H_i\rangle$, then $\langle H, H_j \rangle \subset \langle H, H_i \rangle$, where $y \in H_j$.
\end{enumerate}
\end{lemma}

We fix the following notation. Let $I_{2g}$ denote the identity matrix of order $2g$, and $E(i,j)$ denote the $2g \times 2g$ matrix whose $(i,j)$th entry is $1$ and all other entries are $0$. Let $P_{2i}=\Psi_k(T_{b_i}^{-1})=I_{2g}+E(2i,2i-1), \text{ and }$, $P_{2i+1}=\Psi_k(T_{c_i}^{-1}T_{a_i}T_{a_{i+1}})=I_{2g}+E(2i+1,2i)+E(2i-1,2i+2),$ and 
let $\mathcal{L}_g= \{T_{a_1},T_{b_1}, \ldots,T_{a_g},T_{b_g},T_{c_1},\ldots,T_{c_{g-1}}\},$ be the Lickorish generating set~\cite{WBRL} for $\Mod(S_g)$ (where the curves are as indicated in Figure~\ref{fig:regular_cover}).

\begin{theorem}
For $k$ prime and $g \geq 1$, $\LMod_{p_k}(S_g)$ is a maximal subgroup of $\Mod(S_g)$. 
\end{theorem}

\begin{proof}
Since $\Mod(S_g)[k] < \LMod_{p_k}(S_g)$, it suffices to show that $\Psi_k(\LMod_{p_k}(S_g))$ is maximal in $\Sp(2g;\Z_k)$. From Theorem~\ref{thm:symp_lift_crit}, the distinct cosets of $\Psi_k(\LMod_{p_k}(S_g))$ in $\Sp(2g,\Z_k)$ are in bijective correspondence with the set $$\{\{\ell v : \ell \in \Z_k^{\times} \}:  v \in H_1(S_g,\Z_k) \text{ is primitive}\}.$$ \noindent Moreover, each coset  is represented by a matrix that maps the vector $e_1 \in H_1(S_g,\Z_k)$ to some vector $\ell v$, where $\ell \in \Z_k^{\times}$. The distinct coset representatives are listed in Table~\ref{tab:coset_reps} below. 

\begin{table}[H]
\begin{tabular}{|l|l|l|}
\hline
Coset type & Primitive vector $v$  & Coset representative \\
\hline
$1$ & $(1,x_{1,2},\cdots,x_{1,2g})$ &  $P_2^{x_{1,2}-1}\Psi(T_{a_1}^{-x_{1,4}(x_{1,3}-1)})P_3^{x_{1,3}-1}P_4^{x_{1,4}-1}\cdots$ \\ 
&  & $\Psi(T_{a_{g-1}}^{-x_{1,2g}(x_{1,2g-1}-1)})P_{2g-1}^{x_{1,2g-1}-1}P_{2g}^{x_{1,2g}}P_{2g-1}P_{2g-2}\cdots P_{2}$ \\ \hline 

$2$ & $(0, 1,x_{2,3},\cdots,x_{2,2g})$ & $ \Psi(T_{a_1}^{-x_{2,4}(x_{2,3}-1)})P_3^{x_{2,3}-1}P_4^{x_{2,4}-1}\cdots \Psi(T_{a_{g-1}}^{-x_{2,2g}(x_{2,2g-1}-1)})$\\
& & $P_{2g-1}^{x_{2,2g-1}-1}P_{2g}^{x_{2,2g}}P_{2g-1}P_{2g-2}\cdots P_{3}\Psi(T_{b_1}T_{a_1}T_{b_1})^{-1}$ \\ \hline

$3$ & $(0,0,1,x_{3,4},\cdots,x_{3,2g})$ & $ P_4^{x_{3,4}-1} \Psi(T_{a_{2}}^{-x_{3,6}(x_{3,5}-1)})P_{5}^{x_{3,5}-1}P_{6}^{x_{3,6}-1} \cdots \Psi(T_{a_{g-1}}^{-x_{3,2g}(x_{3,2g-1}-1)})$ \\
& & $P_{2g-1}^{x_{3,2g-1}-1}P_{2g}^{x_{3,2g}}P_{2g-1}P_{2g-2}\cdots P_{4}\Psi(T_{a_1}T_{b_1}T_{c_1}T_{b_2}T_{a_2})^{-2}$ \\ \hline

$\vdots$ & $\vdots$ & $\vdots$ \\ \hline

$2g-1$ & $(0,0,\cdots,1,x_{2g-1,2g})$ & $P_{2g}^{x_{2g-1,2g}}\Psi(T_{a_1}T_{b_1}T_{c_1}T_{b_2}T_{c_2}\cdots T_{c_{g-1}}T_{b_g}T_{a_g})^{-2}$\\ \hline

$2g$ & $(0,0,\cdots,0,1)$ & $\Psi((T_{b_g}T_{a_g}T_{b_g})^{-1}(T_{a_1}T_{b_1}T_{c_1}T_{b_2}T_{c_2}\cdots T_{c_{g-1}}T_{b_g}T_{a_g})^{-2})$ \\ \hline 
\end{tabular}
\caption{The coset representatives of $\Psi_k(\LMod_{p_k}(S_g))$ in $\Sp(2g,\Z_k)$.}
\label{tab:coset_reps}
\end{table}
For $ 1\leq i \leq 2g$, let $g_{i,j}$ (for $1 \leq j \leq k^{2g-i}$) denote the distinct coset representatives of type $i$ with $g_{1,1}$ being the identity matrix. Let $H:=\Psi_k(\LMod_{p_k}(S_g))$ and let $H_{i,j} = \langle H, g_{i,j} \rangle$. In the view of Lemma~\ref{lem:imp_max}, it suffices to show that $\langle H, g_{i,j}\rangle = \Sp(2g,\Z_k)$ whenever $(i,j) \neq (1,1)$. Since $\Psi_k(h) \in H$, for each $h \in \mathcal{L}_g \setminus \{T_{b_1}\}$, this amounts to showing that $\Psi_k(T_{b_1})\in H_{i,j}$ for all $(i,j)\neq (1,1)$. 

First, we consider the case when $i = 2$. It is apparent that $H_{2,j}=\langle H, \Psi_k(T_{b_1}T_{a_1}T_{b_1})^{-1}\rangle,$ for each $j$. By the braid relation and fact that $\Psi(T_{a_1})\in H$, we have $\Psi({T_{b_1}})\in \langle H, g_{2,j}\rangle$ for each $j$.
 
Now we consider the case $i\geq 3$. It follows immediately that $$H_{i,j} = \langle H, \Psi(T_{a_1}T_{b_1}T_{c_1}T_{b_2}T_{c_2}\cdots T_{c_{\ell-1}}T_{b_\ell}T_{a_\ell})^{-2} \rangle$$ for each $j$, where $\ell=\lfloor \frac{i+1}{2}\rfloor$. Since $\Psi(T_{a_1}T_{b_1}T_{c_1}T_{b_2}T_{c_2}\cdots T_{c_{\ell-1}}T_{b_\ell}T_{a_\ell})^{2}\in H_{i,j}$, the commutativity relation between Dehn twists (about disjoint curves), would imply that $\Psi(T_{b_1}T_{c_1}T_{a_1}T_{b_1})\in H_{i,j}$. Thus, we have
\begin{eqnarray*}
\Psi((T_{b_1}T_{c_1}T_{a_1}T_{b_1})^{-1}T_{b_2}(T_{b_1}T_{c_1}T_{a_1}T_{b_1})) & = & \Psi(T_{b_1}^{-1}T_{c_1}^{-1}T_{b_2}T_{c_1}T_{b_1}) \\
& = & \Psi(T_{b_1}^{-1}T_{b_2}T_{c_1}T_{b_2}^{-1}T_{b_1}) \\
& = & \Psi(T_{b_2}T_{c_1}T_{b_1}T_{c_1}^{-1}T_{b_2}^{-1}),
\end{eqnarray*}
where the first equality follows from the commutativity (of Dehn twists about disjoint curves) and the last two follow from the braid relation between Dehn twists about curves (that intersect once). 
Hence, $\Psi({T_{b_1}})\in H_{i,j}$ for each $i \geq 3$.

Finally, we consider the case $i = 1$. In order to show that $\Psi(T_{b_1}) \in H_{1,j}$ (for $j 
\neq 1$), we will show that $H_{1,j}$ will contain an element of $g_{i,j'}H$ for some $i\geq 2$. Suppose that $g_{1,j}$ corresponds to the primitive vector $(1,y_2,y_3,\cdots,y_{2g})$. Choose a minimum $n$ such that $y_n\neq 0$. If ${\bar{y}}_n y_n\equiv 1 \pmod k$, then the matrix $M=I_{2g}+({\bar{y}}_n-1)E(1,1)+(y_n-1)E(2,2)  \in H$. Now setting $Q_2 = \Psi(T_{a_1})$, consider the following matrices for $1 \le i \le g-1$:
\begin{eqnarray*}
	Q_{2i+2} &  = & \Psi(T_{d_i} T_{a_1}^{-1} T_{a_{i+1}}^{-1})^{-1} \\
                   & = & I_{2g} + E(2i+1, 2) + E(1, 2i+2) \text{ and } \\
    Q_{2i+1} & =  & \Psi(T_{a_{i+1}} T_{b_{i+1}} T_{a_{i+1}}) ~ Q_{2i+2} ~ \Psi(T_{a_{i+1}} T_{b_{i+1}} T_{a_{i+1}})^{-1} \\
	               & = & I_{2g} - E(2i+2, 2) + E(1, 2i+1),
	\end{eqnarray*} 
where the $d_i$ are as shown in Figure~\ref{fig:curvesdi}. Then $\tilde{g}_{1,j}=Q_n^{-\bar{y}_n}g_{1,j}M \in \langle H, g_{n,j'}\rangle$. Thus, our assertion now follows from Lemma~\ref{lem:imp_max}.
\end{proof}

\subsection{$\LMod_{p_k}(S_g)$ is self-normalizing}
\label{sec:self_normal}
A subgroup $H$ of a group $G$ is said to be \textit{self-normalizing} if the normalizer $N_G(H) = H$. In other words, this means that if $H \trianglelefteq K \le G$ then $H=K$. In this subsection, we show that $\LMod_{p_k}(S_g)$ satisfies this property. 

\begin{theorem}
For $g \geq 2$ and $k \geq 2$, $\LMod_{p_k}(S_g)$ is self-normalizing in $\Mod(S_g)$. 
\end{theorem}

\begin{proof}
Let us assume on the contrary that there exists a proper subgroup $H < \Mod(S_g)$ in which $\LMod_{p_k}(S_g)$ is normal.  Now choose an $h \in H \setminus \LMod_{p_k}(S_g)$. Then by our assumption $h^{-1} f h \in \LMod_{p_k}(S_g)$ for any $f \in \LMod_{p_k}(S_g)$. By Theorem~\ref{thm:symp_lift_crit}, this is equivalent to requiring that 
$(h^{-1} f h)_{\#} (e_1) = \ell e_1$, for some $\ell \in \Z_k^{\times}$. In other words, $(fh)_{\#}(e_1) = \ell (h_{\#}(e_1))$. Denoting $h_{\#}(e_1) = v$, we have $f_{\#}(v) = \ell \, v$, for some $\ell \in \Z_k^{\times}$. Thus, it suffices to show that for any primitive vector $v \in H_1(S_g,\Z_k)$ such that $v \notin \{\ell e_1 : \ell \in \Z_k^\times\}$, there exists an $f \in \LMod_{p_k}(S_g)$ such that $f_{\#}(v) \neq \ell v$, for any $\ell \in \Z_k^{\times}$.

Now given a $v \in H_1(S_g,\Z_k)$, consider a $\bar{v} \in H_1(S_g,\Z)$ such that $v = \bar{v} \pmod k$. Let 
$\bar{v} = \sum_{i=1}^g \alpha_i a_i + \sum_{i=1}^g \beta_i b_i$. We break our argument into the following two cases.

\textit{Case 1: At least one among $\{\alpha_2, \beta_2, \ldots, \alpha_g, \beta_g\}$ is nonzero.} Without loss of generality, let us assume that $\alpha_2 \neq 0$. Then $\bar{v} = \alpha_1a_1 + \beta_1b_1 + mw,$ where $m \in \Z$ and $w \in \langle a_2,b_2, \ldots, a_g,b_g\rangle$ is primitive. Let $c$ be an oriented simple closed curve in $S_g$ that represents the class of $w$. Note that $c$ can be chosen to be disjoint from both $a_1$ and $b_1$. Let $f \in \Mod(S_g)$ such that $f(a_1) = a_1$, $f(c) = b_2$, and $f$ is represented by a homeomorphism that restricts to the identity in a closed regular neighborhood of $a_1 \cup b_1$. (Such a choice of $f$ is possible from the change of coordinates principle.) Then $f_{\#}(\bar{v}) = \alpha_1 a_1 +\beta_1b_1 + mb_2$, from which it is apparent that $f_{\#}(\bar{v}) \neq \ell \bar{v}$, for any $\ell \in \Z_k^{\times}$. The argument for any other $\alpha_i$ or $\beta_j$ would be similar. 

\textit{Case 2: $\alpha_i = \beta_i = 0$, for $2 \leq i \leq g$.} We may further assume that $\beta_1 \neq 0$, for this would  otherwise imply that $h \in \LMod_{p_k}(S_g)$. In this case, we can see that $(T_{c_1})_{\#}(v) =\alpha_1a_1+\beta_1 (b_1+a_1-a_2) \neq \ell v$, for any $\ell \in \Z_k^{\times}$.
Thus, our assertion follows.
\end{proof}

	\section{Generating $\LMod_{p_k}(S_g)$}
	\label{sec:gen_lmod}

A crucial tool which we will use comes from the work of Johnson in \cite{John1}. We begin by deriving a generating set for $\Psi(\Mod_{p_k}(S_g,e_1))$ using the symplectic criteria for liftability and employing a few number-theoretic arguments. 

\subsection{Generating $\Psi(\Mod_{p_k}(S_g,e_1))$}
We establish the following theorem in this subsection.
	\begin{theorem}\label{thm1}
		$\Psi(\Mod_{p_k}(S_g,e_1))$ is generated by $$\mathbb{S}^{(g)} = \eta_g(\Gamma_1(k)) \cup \Psi( \{ T_{a_2}, T_{b_2}, \dots ,T_{a_g}, T_{b_g}, T_{c_1}, \dots ,T_{c_{g-1}} \})\,$$ where $\eta_g$ is the canonical embedding $\eta_g : \SL(2,\Z) \to \Sp(2g,\mathbb{Z})$ defined by 
$$A \xmapsto{\eta_g}  \begin{bmatrix}
A       & O_{2,2g-2}\\
O_{2g-2,2} & I_{2g-2}
\end{bmatrix},$$ where $O_{m,n}$ is the $m \times n$ zero block and $I_{2g-2}$ is the $(2g-2) \times (2g-2)$ identity block.
	\end{theorem}
	
\noindent The proof of Theorem~\ref{thm1} will require the following technical number-theoretic lemmas. 
\begin{lemma}\label{lemma1}
Let $\mathcal{A}$ be a tuple $(a_{11}, \dots ,a_{(2g)1}) \in \Z^{2g}$ such that $\gcd(a_{11}, \dots ,a_{(2g)1})=1$. Then there exists $\mathcal{B} = (\lambda_1, \dots , \lambda_{2g}) \in \Z^{2g}$ such that $\gcd(\lambda_1 , \lambda_2)=1$ and $\mathcal{A}\cdot \mathcal{B} = 1$. Further, if $a_{11} \equiv 1 \pmod k$ and $a_{j1} \equiv 0 \pmod k$ for all $2 \le j \le 2g$, then $\lambda_1 \equiv 1 \pmod k$.
	\end{lemma}
	
\begin{proof}
Consider $\delta = \gcd (a_{21}, a_{31}, \dots , a_{(2g)1})$. Since $\gcd (a_{11}, \delta) = 1$, we have integers $t_1, t_2$ such that $t_1 a_{11} + t_2 \delta = 1$. We define $\gamma = \gcd (a_{31} , \dots , a_{(2g)1})$. So we have integers $t_3, \dots, t_{2g}$ such that $t_3 a_{31} + \dots + t_{2g} a_{(2g)1} = \gamma$. As $\delta = \gcd(a_{21}, \gamma)$, we can choose integers $r_1, r_2$ such that $r_1 a_{21} + r_2 \gamma = \delta$. This equation can be rewritten as 
$$r_1 \frac{a_{21}}{\delta} + r_2 \frac{\gamma}{\delta} = 1.$$
		
		By Dirichlet prime number theorem, there exists a natural number $N$ such that $r_1 + N \frac{\gamma}{\delta}$ is a prime number bigger than $t_1$. Denoting $\beta_1 = r_1 + N \frac{\gamma}{\delta}$ and $\beta_2 = r_2 - N \frac{a_{21}}{\delta}$, we have:	
		$$\beta_1 \frac{a_{21}}{\delta} + \beta_2 \frac{\gamma}{\delta} = 1.$$
\noindent From the equations above, we deduce that $t_1 a_{11} + t_2 (\beta_1 a_{21} + \beta_2 \gamma) = 1$, which implies that
		\begin{equation} 
		\label{eqn:1}
		t_1 a_{11} + t_2 (\beta_1 a_{21} + \beta_2 (t_3 a_{31} + \dots + t_{2g} a_{(2g)1})) = 1.
		\end{equation}

Denoting $\lambda_1 = t_1, \lambda_2 = t_2 \beta_1, \lambda_i = t_2 \beta_2 t_i$, for $3 \le i \le 2g$, we further assume that $a_{11} \equiv 1 \pmod k$ and $a_{j1} \equiv 0 \pmod k$, for $2 \le j \le 2g$. Thus, Equation~\ref{eqn:1} takes the form $\lambda_1 a_{11} + \lambda_2 a_{21} + \dots + \lambda_{2g} a_{(2g)1} = 1$, from which it follows that $\lambda_1 \equiv 1 \pmod k$.
	\end{proof}
		
	\begin{lemma}\label{lemma2}
Let $\lambda_1, \lambda_2$ be two coprime integers and $\lambda_1 \equiv 1 \pmod k$. Then there exist $\alpha_1, \alpha_2 \in \Z$ with $\alpha_2 \equiv 0 \pmod k$, such that $\alpha_1 \lambda_1 + \alpha_2 \lambda_2 = 1.$ 
	\end{lemma}
	
\begin{proof}
Since $\gcd(\lambda_1, \lambda_2) = 1$, there exist  $\alpha_1, \alpha_2 \in \Z$ such that $\alpha_1 \lambda_1 + \alpha_2 \lambda_2 = 1.$ Now let us assume that $\lambda_1 \equiv 1 \pmod k$. Consider the set $\{ \lambda_2 k i ~ | ~ 0 \le i \le \lambda_1 - 1 \}$. 

We first establish that $\lambda_1$ divides $(\lambda_2 k i + 1)$, for some $i$. Suppose we assume on the contrary. Then by pigeonhole principle, there exist $0 \le i \ne j \le \lambda_1 - 1$ such that $\lambda_2 k i + 1 \equiv \lambda_2 k j + 1 \pmod {\lambda_1}$, and hence $\lambda_2 k i \equiv \lambda_2 k j \pmod {\lambda_1}$. Since both $\lambda_2$ and $k$ are coprime with $\lambda_1$, we have $i \equiv j \pmod {\lambda_1}$, which is a contradiction. 

Hence, there exists $i_0$ such that $\lambda_1$ divides $(\lambda_2 k i_0 + 1)$. Finally, by choosing $\alpha_1 = (\lambda_2 k i_0 + 1) / \lambda_1$ and $\alpha_2 = - k i_0$, our assertion follows.
	\end{proof}

Let $\Stab_{\Mod(S_g)}(\overrightarrow{a_1})$ denote the subgroup of $\Mod(S_g)$ which preserves the curve $a_1$ with a chosen orientation. Let $\PMod(S)$ denote the pure mapping class group of $S$. We need the following lemma, which follows from \cite[Chapter 4]{FM}.

\begin{lemma}\label{lem:staba1}
	$\Stab_{\Mod(S_g)}(\overrightarrow{a_1}) = \langle  T_{a_1}, T_{c_1}, \dots , T_{c_{g-1}}, T_{a_2},T_{b_2}, \dots , T_{a_g},T_{b_g} \rangle$.
\end{lemma}

\begin{proof}
	Let $S_g - a_1$ denote the surface obtained from $\overline{S_g \ \backslash \ a_1}$ after capping its boundary components with marked disks. We have the following short exact sequence
	\begin{equation} 
		\label{ses1}
		1 \to \Z \to \Stab_{\Mod(S_g)}(\overrightarrow{a_1}) \xrightarrow \pi \PMod(S_g - a_1)  \to 1.
	\end{equation} Since $\PMod(S_g - a_1) = \langle S \rangle$, where
	$$S =  \{ T_{c_1}, \dots , T_{c_{g-1}}, T_{a_2},T_{b_2}, \dots , T_{a_g},T_{b_g} \} ,$$
	if follows from the sequence (\ref{ses1}) that $\Stab_{\Mod(S_g)}(\overrightarrow{a_1}) = \langle S , T_{a_1}\rangle$. 
\end{proof}
	
We are now in a position to prove Theorem~\ref{thm1}.

\begin{proof}[Proof (of Theorem~\ref{thm1})]
	Given an $A = (a_{ij}) \in \Psi(\Mod_{p_k}(S_g,e_1))$, we wish to express $A$ as a product of elements in the generating set $$\mathbb{S}^{(g)} = \eta_g(\Gamma_1(k)) \cup \Psi( \{ T_{a_2}, T_{b_2}, \dots ,T_{a_g}, T_{b_g}, T_{c_1}, \dots ,T_{c_{g-1}} \}).$$ To this end, we iteratively multiply $A$ with matrices from $\mathbb{S}^{(g)}$ so that the end product is the identity matrix $I_{2g}$. For this purpose let us consider the curves $d_1, d_2, \dots, d_{g-1}$ as shown in Figure~\ref{fig:curvesdi}. By Lemma~\ref{lem:staba1}, $\Psi(T_{d_i}) \in \langle \mathbb{S}^{(g)} \rangle$, for all $i$.

	\begin{figure}[h]
		\labellist
		\small
		\pinlabel $d_1$ at 68 50
		\pinlabel $d_2$ at 115 45
		\pinlabel $d_3$ at 160 41
		\pinlabel $d_{g-1}$ at 292 28
		
		\endlabellist
		\centering
		\includegraphics[width=60ex]{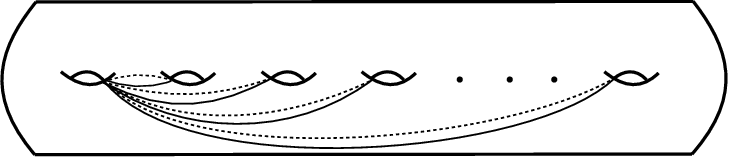}
		\caption{The curves $d_i$ in $S_g$.}
		\label{fig:curvesdi}
	\end{figure}
	
	To begin with, we consider the first column $A(1) = (a_{11}, a_{21}, \dots ,a_{(2g)1})$ of $A$. Lemma \ref{lemma1} provides us a $B = (\lambda_1, \lambda_2,\dots,\lambda_{2g}) \in \Z^{2g}$ such that $\gcd(\lambda_1 , \lambda_2)=1$ and $A(1)\cdot B = 1$. By Lemma \ref{lemma2}, we have $\alpha_1, \alpha_2 \in \Z$, such that $\alpha_1 \equiv 0 \pmod k$ and $\lambda_1\alpha_2 + \lambda_2\alpha_1 = 1.$
	
	Let $I_{2g}$ denote the identity matrix of order $2g$, and $E(i,j)$ denote the $2g \times 2g$ matrix whose $(i,j)$th entry is $1$ and all other entries are $0$. Consider the following matrices, for $1 \le i \le g-1$:
	
	\begin{eqnarray*}
	M_1 & = &\eta_g \left( \begin{bmatrix}
		\lambda_1&\lambda_2 \\
		-\alpha_1&\alpha_2 
	\end{bmatrix} \right),\\
	M_{2i+1} &  = & \Psi(T_{d_i} T_{a_1}^{-1} T_{a_{i+1}}^{-1})^{-1} \\
                   & = & I_{2g} + E(2i+1, 2) + E(1, 2i+2), \text{ and } \\
    M_{2i+2} & =  & \Psi(T_{a_{i+1}} T_{b_{i+1}} T_{a_{i+1}}) ~ M_{2i+1} ~ \Psi(T_{a_{i+1}} T_{b_{i+1}} T_{a_{i+1}})^{-1} \\
	               & = & I_{2g} - E(2i+2, 2) + E(1, 2i+1).
	\end{eqnarray*} 

\noindent Note that $M_1, \Psi(T_{a_1}), \Psi(T_{b_1})^k \in \eta_g(\Gamma_1(k))$. We consider the element
$$A_1= \Psi(T_{a_1})^{- \sum_{i=1}^{g-1} \lambda_{2i+1} \lambda_{2i+2}} M_{2g}^{\lambda_{2g-1}} M_{2g-1}^{\lambda_{2g}} \dots M_4^{\lambda_3} M_3^{\lambda_4} M_1 A.$$

\noindent Then the $(1,1)^{th}$ entry in $A_1$, which we denote as $A_1(1,1)$, equals $1$. Now we consider the following matrices, for $1 \le i \le g-1$: 

\begin{eqnarray*}
N_{2i+1} & = & \Psi(T_{b_1}^{-k} T_{a_{i+1}}^{-k}) M_{2i+1} \Psi(T_{b_1})^k M_{2i+1}^{-1} \\ 
	          & = & I_{2g} - k.E(2i+1, 1) + k.E(2, 2i+2) \text{ and} \\
N_{2i+2} &  = & \Psi(T_{a_{i+1}} T_{b_{i+1}} T_{a_{i+1}}) ~ N_{2i+1} ~ \Psi(T_{a_{i+1}} T_{b_{i+1}} T_{a_{i+1}})^{-1} \\
	          & =  & I_{2g} + k.E(2i+2, 1) + k.E(2, 2i+1).
\end{eqnarray*}

\noindent Next, we consider 
\begin{eqnarray*} 
A_2 & = & N_{2g}^{(-1)^{2g-1} \frac{A_1(2g,1)}{k}} N_{2g-1}^{(-1)^{2g-2} \frac{A_1(2g-1,1)}{k}} \dots N_4^{(-1)^3 \frac{A_1(4,1)}{k}} N_3^{(-1)^2 \frac{A_1(3,1)}{k}} A_1 \text{ and}\\
A_3 & = &\Psi(T_{b_1})^{A_2(2,1)} A_2.
\end{eqnarray*} 
 We note that the first column of $A_3$ is $(1,0,\dots, 0)$, and as $A_3 \in \Sp(2g, \Z)$, the second row of $A_3$ is $(0,1,0,\dots,0)$.

Finally, we consider
\begin{eqnarray*}  
A_4 & = & M_{2g}^{(-1)^{2g} A_3(2g,2)} M_{2g-1}^{(-1)^{2g-1} A_3(2g-1,2)} \dots M_4^{(-1)^4 A_3(4,2)} M_3^{(-1)^3 A_3(3,2)} A_3 \text{ and} \\ 
A_5 & = & {\Psi(T_{a_1})}^{-A_4(1,2)}A_4.
\end{eqnarray*} 
\noindent We see that $A_5$ has the form $\begin{bmatrix}
	I_2 & O_{2 \times 2g-2}\\
	O_{2g-2 \times 2} & M,
\end{bmatrix},$
where $M \in \Sp(2g-2, \Z)$, and  $A_5$ is a product of elements from
$$ \Psi(\{ T_{a_2},\dots,T_{a_g}, T_{b_2}, \dots, T_{b_g}, T_{c_2}, \dots, T_{c_{g-1}} \}).$$
Hence, the theorem follows.
\end{proof}

\subsection{Generating the kernel of $\, \Psi\vert_{\Mod_{p_k}(S_g,e_1)}$}\label{genkernel} We begin this subsection by recalling some relevant definitions and tools from \cite{John1}. 
	
\begin{definition}
For $n \geq 2$, a \textit{chain} in $S_g$ is an ordered collection $(\gamma_1, \gamma_2, \dots , \gamma_n)$ of
oriented simple closed curves such that:
\begin{enumerate}[(i)]
\item for each $1 \le i \le n-1$, the geometric intersection number $i(\gamma_i, \gamma_{i+1}) = 1$ and the algebraic intersection number $\hat i (\gamma_i, \gamma_{i+1}) = +1$,
\item $i(\gamma_i, \gamma_j) = 0$ if $|i-j| > 1$, and
\item the homology classes of the $\gamma_i$'s are linearly independent.	
\end{enumerate}
\noindent The integer $n$ is called the \textit{length} of the above chain $(\gamma_1, \gamma_2, \dots , \gamma_n)$. 
\end{definition}

Given a chain $C = (\gamma_1, \dots , \gamma_{2k+1})$ of odd length in $S_g$, there exists a unique pair of curves that constitutes the boundary of a closed regular neighborhood of $C$. We call this pair $B$ as the \textit{bounding pair associated with the chain $C$.} Moreover, the unique bounding pair map in $\I(S_g)$ associated with such a $B$ is given by the product of opposite twists about the two boundary curves in $B$ with the positive twist being taken about the curve lying to the left of $\gamma_1, \gamma_2, \dots, \gamma_{2k+1}$. This is called the \textit{bounding pair map associated with $C$}, and is denoted by $ch(\gamma_1, \dots , \gamma_{2k+1})$.

We further extend this definition in the following manner. By an \textit{ordered subcollection} of $C$, we mean a subtuple $(\gamma_{i_1}, \gamma_{i_2}, \ldots, \gamma_{i_{\ell}})$ of $C$ such that $\ell$ is even. Given an ordered subcollection $(\gamma_{i_1}, \gamma_{i_2}, \ldots, \gamma_{i_{\ell}})$ of $C$, let $\theta_j$ denote the free homotopy class of the curve $\tilde{\gamma}_{i_j}*\tilde{\gamma}_{i_j+1}*\ldots *\tilde{\gamma}_{i_{(j+1)}-1}$, where $\tilde{\gamma}_m$ denotes the representative of $\gamma_m$ in $(\pi_1(S_g),*)$. We define $$\widehat{ch}(\gamma_{i_1}, \gamma_{i_2}, \ldots, \gamma_{i_{\ell}}):= ch(\theta_1,\ldots,\theta_{\ell-1})$$ to be the \textit{bounding pair map associated with the subcollection $(\gamma_{i_1}, \gamma_{i_2}, \ldots, \gamma_{i_\ell})$} and $(\theta_1,\ldots,\theta_{\ell-1})$ will be called the \textit{chain induced by} $(\gamma_{i_1}, \gamma_{i_2}, \ldots, \gamma_{i_{\ell}})$.  For further details, we refer the reader to \cite[Section 2]{John1}.

 Now we consider the ordered collection $C_g = (a_1,b_1,c_1,b_2,c_2,b_3,\ldots,c_{g-1},b_g,a_g)$ of length $2g+1$, where the curves $a_i$, $b_i$, and $c_i$ are as shown in Figure~\ref{fig:Chain_curve}. For simplicity, we rename these curves and fix the notation $C_g := (\alpha_1,\ldots,\alpha_{2g+1})$. Also consider the curve $\beta$ as shown in Figure~\ref{fig:Chain_curve}.

\begin{figure}[H]
	\labellist
	\small
	\pinlabel $a_1$ at 64 32
	\pinlabel $b_1$ at 61 94
	\pinlabel $c_1$ at 108 52
	\pinlabel $b_2$ at 125 90
	\pinlabel $c_2$ at 194 51
	\pinlabel $b_3$ at 223 94
	\pinlabel $a_{g}$ at 432 32
	\pinlabel $b_{g}$ at 426 94
	\pinlabel $\beta$ at 147 123
	\endlabellist
	\begin{center}
		\includegraphics[width=75ex]{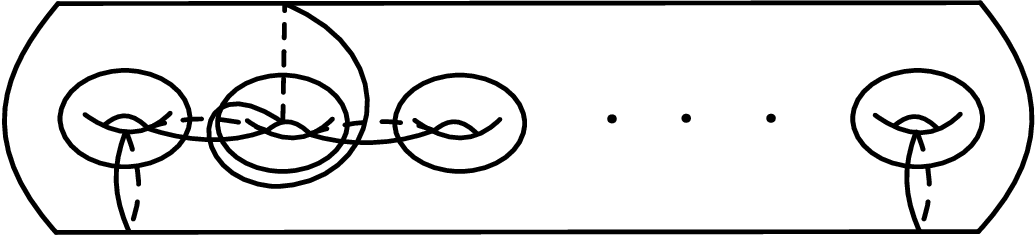}
	\end{center}
	\caption{The Johnson chain of curves.}
	\label{fig:Chain_curve}
\end{figure}

\noindent We now state the main result from \cite{John1} that is relevant to our context.
 
\begin{theorem}[Johnson]
\label{thm:johnson}
For $g \ge 3$, the bounding pair maps associated with the ordered subcollections of $C_g$ and $(\beta, \alpha_5, \alpha_6, \ldots, \alpha_{2g+1})$ generate $\I(S_g)$.
\end{theorem}

Let the \textit{genus} of a bounding pair map associated with a chain $C$ of length $2k+1$ is $k$. Following the notation from Section~\ref{sec:intro}, we now establish the main result in this subsection. 
\begin{theorem}
\label{thm:gen_kernel}
For $k \geq 2$ and $g \geq 3$, 
$$ \ker (\Psi\vert_{\Mod_{p_k}(S_g,e_1)}) \subset \langle \S_{g,k} \cup \{T_{a_1},T_{a_2},T_{b_2}, \ldots,T_{a_g},T_{b_g},T_{c_1},\ldots,T_{c_{g-1}}\} \rangle,$$
where $\S_{g,k} = \{F_1,\ldots,F_{g-2}\}$ is a collection of bounding pair maps such that for each $i$, $F_i$ is a bounding pair map associated with an ordered subcollection of $C_g$ of genus $i$ that intersect $a_1$ nontrivially (as shown in Figure~\ref{fig:torelli_gens}).
\end{theorem}

\begin{proof}
By construction, a bounding pair map $F$ associated with any ordered subcollection of $(\beta, \alpha_5, \ldots, \alpha_{2g+1})$ or an ordered subcollection of $C_g$ that is disjoint from $a_1$ (as stated in Theorem~\ref{thm:johnson}),  will preserve $a_1$ with orientation. Thus, by Lemma~\ref{lem:staba1}, it follows that $F \in \langle T_{a_1}, T_{c_1}, \dots , T_{c_{g-1}}, T_{a_2}, T_{b_2}, \dots , T_{a_g},T_{b_g}\rangle$. 

For $i = 1,2$,  let $P_i = (\alpha_{i1},\ldots,\alpha_{i\ell})$ be an ordered subcollection of $C_g$ that intersects $a_1$. Then there exists a $G \in \Stab_{\Mod(S_g)}(\overrightarrow{a_1})$ that maps the chain induced by the ordered subcollection $(\alpha_{11},\ldots,\alpha_{1\ell})$ to the chain induced by the ordered subcollection $(\alpha_{21},\ldots,\alpha_{2\ell})$. Consequently, $\widehat{ch}(P_1)$ is conjugate to $\widehat{ch}(P_2)$ via $G$.

Hence, it follows that any two bounding pair maps associated with ordered subcollections of $C_g$ of the same length which intersect $a_1$ are conjugate by an element in $$\Stab_{\Mod(S_g)}(\overrightarrow{a_1}) = \langle T_{a_1}, T_{c_1}, \dots , T_{c_{g-1}}, T_{a_2}, T_{b_2}, \dots , T_{a_g},T_{b_g}\rangle,$$ and our assertion follows.
\end{proof}

\subsection{Generating $\LMod_{p_k}(S_g)$}
Following the notation from Section~\ref{sec:intro}, we obtain the next result as a direct consequence of Theorems~\ref{thm1} and~\ref{thm:gen_kernel}.
\begin{corollary}
\label{cor:gen_modsge1}
For $k \geq 2$ and $g \geq 3$, 
$$\Mod_{p_k}(S_g,e_1) = \langle \ \tilde{\S}_k \cup \S_{g,k} \cup \{T_{a_2},T_{b_2}, \ldots,T_{a_g},T_{b_g},T_{c_1},\ldots,T_{c_{g-1}}\}\rangle.$$
\end{corollary} 

\noindent Furthermore, from Theorem~\ref{thm:norm_series_gen} and Corollary~\ref{cor:gen_modsge1}, we obtain our main result. 

\begin{theorem}[Main theorem]
\label{thm:gen_modsge1}
For $k \geq 2$ and $g \geq 3$, 
$$\LMod_{p_k}(S_g) = \langle \ \S_k'' \cup \tilde{\S}_k \cup \S_{g,k} \cup \{T_{a_2},T_{b_2}, \ldots,T_{a_g},T_{b_g},T_{c_1},\ldots,T_{c_{g-1}}\}\rangle.$$
\end{theorem} 

\subsection{Generating $\SMod_{p_k}(S_{g_k})$} 
\label{sec:central}
Recalling the notation from Section~\ref{sec:intro}, let $R_{g,k} \in \Mod(S_{g_k})$ be represented by the free $2\pi/k$ rotation of $S_{g_k}$ that induces the cover $p_k$. Let $\SMod_{p_k}(S_{g_k})$ be the \textit{symmetric mapping class group} of $p_k$, which is the normalizer of $\langle R_{g,k} \rangle$ in $\Mod(S_{g_k})$. By Birman-Hilden theory, we have the following short exact sequence for $g \geq 2$:
$$1 \to \langle R_{g,k} \rangle \cong \Z_k \to \SMod_{p_k}(S_{g_k}) \to \LMod_{p_k}(S_g) \to 1.$$

By pulling back the generators of $\LMod_{p_k}(S_g)$ in Theorem~\ref{thm:gen_modsge1}, we obtain an explicit generating set $G_{g,k}$ for the normalizer $N_{\Mod(S_{g_k})}(\langle R_{g,k} \rangle)$ of $\langle R_{g,k} \rangle$ in $\Mod(S_{g_k})$. Moreover, as the centralizer $C_{\Mod(S_{g_k})}(R_{g,k})$ of $R_{g,k}$ is a finite-index subgroup of  $N_{\Mod(S_{g_k})}(\langle R_{g,k} \rangle)$, one can derive a finite generating set for $C_{\Mod(S_{g_k})}(R_{g,k})$ using $G_{g,k}$, which is in general cumbersome. In this subsection, we derive an explicit finite generating set for the centralizer $C_{\Mod(S_{g_k})}(\langle R_{g,k} \rangle)$ of $\langle R_{g,k} \rangle$ in $\Mod(S_{g_k})$ when $k=2$. 

\begin{corollary}
\label{thm:cent_gen_set}
For $g \geq 3$, consider the two-sheeted regular cover $p_2 : S_{2g-1} \to S_g$. Then: 
\begin{enumerate}[(i)]
\item $\LMod_{p_2}(S_g) = \langle \{ T_{a_1}, T_{b_1}^2, T_{a_2},T_{b_2}, \ldots,T_{a_g},T_{b_g},T_{c_1},\ldots,T_{c_{g-1}}\}\rangle$ and 
\item $\SMod_{p_2}(S_{2g-1}) = C_{\Mod(S_{2g-1})}(\langle R_{g,2}\rangle)=\langle \{R_{g,2}, T_{a_{1,1}}T_{a_{2,1}}, T_{b_{1,1}}, T_{a_{1,2}}T_{a_{2,2}}, \\ T_{b_{1,2}}T_{b_{2,2}}, \ldots,T_{a_{1,g}}T_{a_{2,g}}, T_{b_{1,g}}T_{b_{2,g}}, T_{c_{1,1}}T_{c_{2,1}}, \ldots, T_{c_{1,g-1}}T_{c_{2,g-1}} \}\rangle,$ where the Dehn twists are about the curves shown in Figure~\ref{fig:2sheetedcover} below. 
\begin{figure}[H]
		\includegraphics[width=85ex]{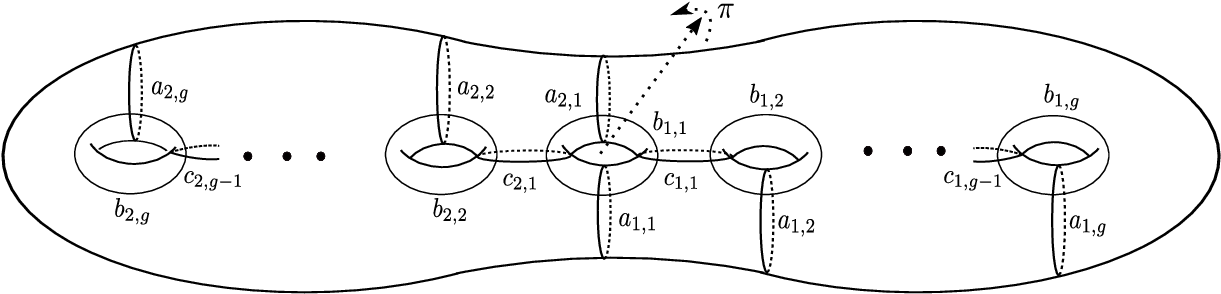}
		\caption{A rotation of $S_{2g-1}$ that induces the cover $p_2$. .}
		\label{fig:2sheetedcover}
	\end{figure}
\end{enumerate}
\end{corollary}

\begin{proof}
For the case $g \geq 3$ and $k=2$, Theorem~\ref{thm:gen_modsge1} would imply that $\LMod_{p_2}(S_g)=\Mod_{p_2}(S_g,e_1)$. So, we have
$$\LMod_{p_2}(S_g) = \langle  \tilde{\S}_2 \cup \S_{g,2} \cup \{T_{a_2},T_{b_2}, \ldots,T_{a_g},T_{b_g},T_{c_1},\ldots,T_{c_{g-1}}\}\rangle.$$ We will now follow the notation from Sections~\ref{sec:intro} and~\ref{sec:gen_lmod}. It is a well known fact that $\S_2=\langle \Psi(T_{a_1}), \Psi(T_{b_1}^2) \rangle$ (see \cite{WS}). Thus, we have $\tilde{\S}_2=\langle T_{a_1}, T_{b_1}^2 \rangle$. By applying the chain relations in $\Mod(S_g)$, we see that $F_i \in \S_{g,2}$ has the form $F_i= T_{d_{i1}}T_{d_{i2}}^{-1}$, where
$$T_{d_{i1}}T_{d_{i2}} = (T_{b_1}T_{c_1}T_{b_2}T_{c_2}\cdots T_{b_i}T_{c_i}T_{b_{i+1}})^{2i+2}=(T_{b_1}^2T_{c_1}T_{b_2}T_{c_2}\cdots T_{b_i}T_{c_i}T_{b_{i+1}})^{2i+1}.$$ Since $T_{d_{i2}}$ is conjugate to $T_{b_1}$ by an element in $\Stab_{\Mod(S_g)}(\overrightarrow{a_1})$, the assertion in (i) follows. 

To show (ii), we first observe that $\SMod_{p_2}(S_{2g-1}) = C_{\Mod(S_{2g-1})}(\langle R_{g,2}\rangle)$. Moreover, the collection of curves $\{a_1,a_2,b_2, \ldots, a_g,b_g,c_1,\ldots,c_{g-1}\}$ in $S_g$ (shown in Figure~\ref{fig:regular_cover}) pulls back under $p_2$ to the collection $\C \setminus \{b_{1,1}\}$, where $\C$ is the collection of curves in $S_{2g-1}$ shown in Figure~\ref{fig:2sheetedcover}. Thus, by pulling back the generators of $\LMod_{p_2}(S_g)$ under the epimorphism $\SMod_{p_2}(S_{2g-1}) \to \LMod_{p_2}(S_g)$, the assertion follows. 
\end{proof}
	
\section{Lifting mapping classes to the infinite ladder surface}
\label{sec:ladder}
	Let $\mathcal{L}$ be the infinite ladder surface, that is, the surface of infinite type with two ends accumulated by genus. In this section, we examine the liftable mapping classes of the regular infinite-sheeted cyclic cover $q_g :  \mathcal{L} \to S_g$  induced by an $\langle h^{g-1} \rangle$-action on $\mathcal{L}$, where $h$ is the standard handle shift map on $\mathcal{L}$. The covering map $q_g$ is illustrated in Figure~\ref{fig:ladder} below (where the `$\approx$' symbol means `homeomorphic to'). As in finite-sheeted case, $(q_g)_*(\pi_1(\L)) \lhd \pi_1(S_g)$ and is identified with the kernel of the map $\phi:\pi_1(S_g)\to \mathbb{Z}$ given by $c\xmapsto{\phi} \hat{i}(a_1,c)$. Consider the induced homomorphism $(q_g)_{\#} : H_1(\mathcal{L},\Z) \to H_1(S_g, \Z)$ given by 
\begin{eqnarray}
\label{eqn:inducedhom}
		\notag \tilde{a}_j & \xmapsto{(q_g)_{\#}} & a_{i},\text{ if } i \equiv (j-2)_{g-1}+2, \\
		\tilde{b}_j & \xmapsto{(q_g)_{\#}} & b_{i},\text{ if } i \equiv (j-2)_{g-1}+2, \\
	 \notag s_j & \xmapsto{(q_g)_{\#}} & a_{1}, \text{ if } j \equiv 0 \pmod{g-1},   
	\end{eqnarray}
	
	\noindent where for any $z,r \in \mathbb{Z}$, $r \ge 1$, $(z)_r$ denotes the unique integer in $\{0,1,\ldots, r-1\}$ such that $(z)_r \equiv z \pmod{r}$.
	
	\begin{figure}[H]
		\labellist
		
		\small
		\pinlabel $q_g$ at 171 145
		\pinlabel $\approx$ at 211 82
		\pinlabel $q_g$ at 244 145
		
		\pinlabel $a_1$ at 1 70
		\pinlabel $b_1$ at 95 94
		\pinlabel $a_2$ at 25 19
		\pinlabel $b_2$ at 70 44
		\pinlabel $a_3$ at 108 5
		\pinlabel $b_3$ at 115 51
		\pinlabel $a_{g}$ at 32 131
		\pinlabel $b_{g}$ at 71 125
		
		\pinlabel $a_1$ at 304 46
		\pinlabel $b_1$ at 326 94
		\pinlabel $a_2$ at 352 46
		\pinlabel $b_2$ at 374 94
		\pinlabel $a_3$ at 402 46
		\pinlabel $b_3$ at 426 93
		\pinlabel $a_{g}$ at 507 47
		\pinlabel $b_{g}$ at 530 94

		\pinlabel $s_{0}$ at 128 166
		\pinlabel $s_{1}$ at 194 166
		\pinlabel $s_2$ at 261 166
		\pinlabel $s_{g-1}$ at 383 166
		
		\pinlabel $h$ at 162 260
		
		\pinlabel $\tilde{a}_{2}$ at 151 181
		\pinlabel $\tilde{b}_{2}$ at 169 234
		\pinlabel $\tilde{a}_{3}$ at 218 181
		\pinlabel $\tilde{b}_{3}$ at 236 234
		\pinlabel $\tilde{a}_{g}$ at 329 181
		\pinlabel $\tilde{b}_{g}$ at 349 234
		
		\endlabellist
		\centering
		\includegraphics[width=80ex]{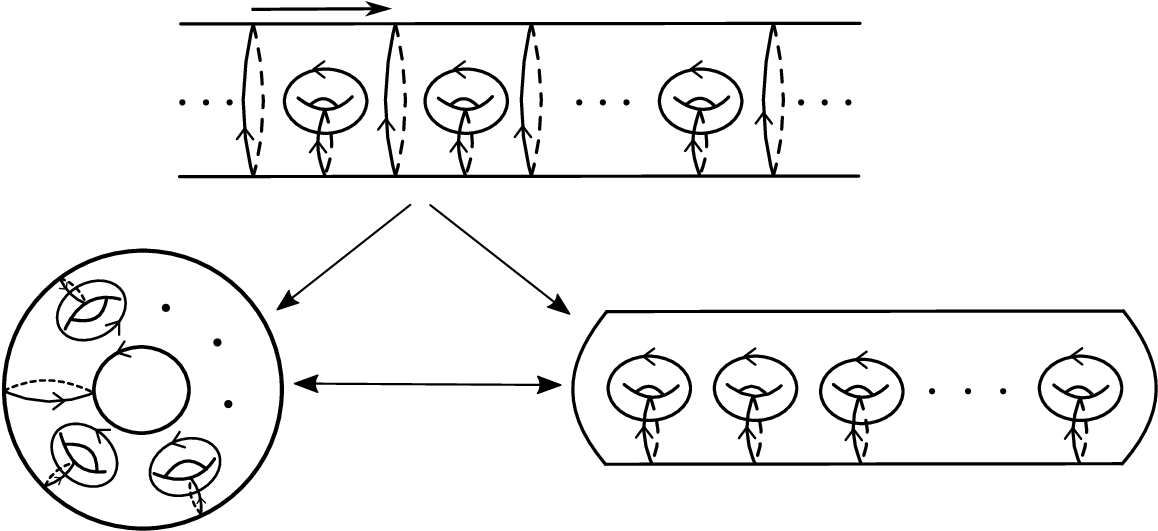}
		\caption{The homomorphism induced $q_g$.}
		\label{fig:ladder}
	\end{figure}
\noindent	Note that for $k  \geq 2$, we have $p_k \circ q_{g_k} = q_g$. This is illustrated for the case when $g=2$ and $k=3$ in Figure~\ref{example}. 
\begin{figure}[H]
		\labellist
		
		\small
		\pinlabel $q_2$ at 237 159
		\pinlabel $q_4$ at 142 162
		\pinlabel $p_3$ at 186 81
		
		\pinlabel $a_1$ at 1 71
		\pinlabel $b_1$ at 76 100
		\pinlabel $a_2$ at 37 10
		\pinlabel $b_2$ at 75 43
		\pinlabel $a_3$ at 150 68
		\pinlabel $b_3$ at 106 94
		\pinlabel $a_{4}$ at 36 133
		\pinlabel $b_{4}$ at 69 128
		
		\pinlabel $a_1$ at 262 58
		\pinlabel $b_1$ at 300 100
		\pinlabel $a_2$ at 379 70
		\pinlabel $b_2$ at 339 95
		
		\pinlabel $s_{0}$ at 83 188
		\pinlabel $s_{1}$ at 152 188
		\pinlabel $s_2$ at 218 188
		\pinlabel $s_{3}$ at 288 188
		
		\pinlabel $h^3$ at 180 284
		\pinlabel $2\pi/3$ at 77 153
		
		\pinlabel $\tilde{a}_{2}$ at 119 187
		\pinlabel $\tilde{b}_{2}$ at 116 257
		\pinlabel $\tilde{a}_{3}$ at 185 187
		\pinlabel $\tilde{b}_{3}$ at 185 257
		\pinlabel $\tilde{a}_{4}$ at 255 187
		\pinlabel $\tilde{b}_{4}$ at 252 257
		
		\endlabellist
		\centering
		\includegraphics[width=55ex]{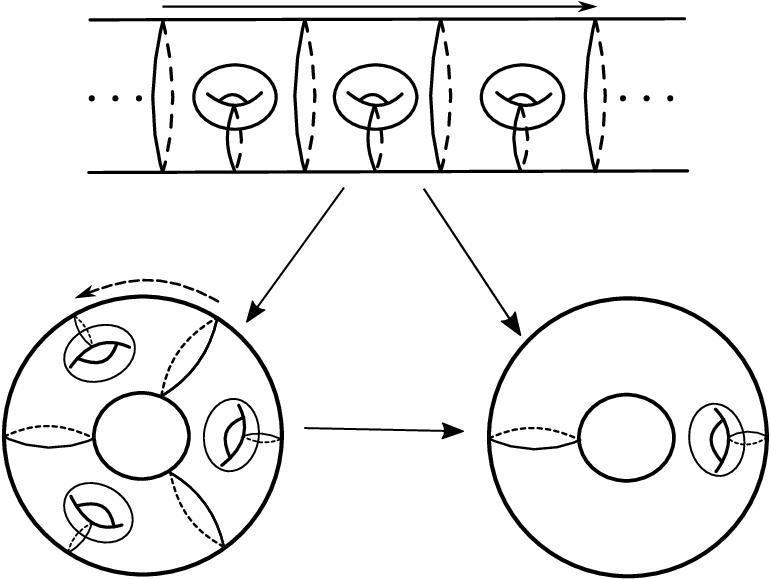}
		\caption{An illustration of $p_3 \circ q_{4} = q_2$.}
		\label{example}
	\end{figure}

\noindent We will use `$\sim$' to denote `isotopic to' and $id$ to denote the identity map. We will now show that the cover $q_g$ satisfies the Birman-Hilden property~\cite{MW}. 		
	\begin{lemma}
	\label{lem:inf_BH_cover}
Suppose that a homeomorphism $\tilde{f} : \mathcal{L} \to \mathcal{L}$ is a lift of a homeomorphism $f:S_g\to S_g$ under $q_g$. If $\tilde{f} \sim id$, then $f \sim id$. 
\end{lemma}
	\begin{proof}
The cover $q_g$ maps the chain of curves $(\tilde{\alpha}_3, \ldots,\tilde{\alpha}_{2g+1})$ in $\L$ to the chain $({\alpha}_3, \ldots,{\alpha}_{2g+1})$ in $S_g$ as shown in Figure~\ref{fig:ladder1}. Let $\tilde{f} \in \Homeo^+(\L)$ such that $\tilde{f} \sim id$. Then for $3 \leq i \leq 2g+1$, we have $\tilde{f}(\tilde{\alpha}_i) \sim \tilde{\alpha}_i$, and so it follows that $f(\alpha_i) \sim \alpha_i$. Since the complement of an open regular neighborhood of the chain $({\alpha}_3, \ldots,{\alpha}_{2g+1})$ in $S_g$ is an annulus containing $\alpha_1$, we have that $f  \sim T_{\alpha_1}^k$, for some $k \in \Z$. Consequently, we have
 $$\tilde{f} \ \sim \ h^m\prod_{\substack{i \in \Z \\ i \equiv 0\, (\text{mod}\,{g-1})}} T_{s_i}^k,$$ for some $m \in \Z$. Since $\tilde{f} \sim id$, we have that $m=0$ and $k=0$. Thus, it follows that $f \sim id$, as required. 
\end{proof}
		
		\begin{figure}[htbp]
			\labellist
			\small
			\pinlabel $q_g$ at 170 125
			
			\pinlabel $\alpha_1$ at 53 18
			\pinlabel $\alpha_2$ at 61 63
			\pinlabel $\alpha_3$ at 89 32
			\pinlabel $\alpha_4$ at 113 63
			
			\pinlabel $\alpha_{2g+1}$ at 282 18
			\pinlabel $\alpha_{2g}$ at 300 64
			
			\pinlabel $s_{0}$ at 60 167
			\pinlabel $s_{1}$ at 127 167
			\pinlabel $s_2$ at 190 167
			\pinlabel $s_{g-1}$ at 305 167
			
			\pinlabel $h$ at 90 260
			
			\pinlabel $\tilde{\alpha}_{3}$ at 81 235
			\pinlabel $\tilde{\alpha}_{4}$ at 93 187
			\pinlabel $\tilde{\alpha}_{5}$ at 137 198
			\pinlabel $\tilde{\alpha}_{6}$ at 158 230
			\pinlabel $\tilde{\alpha}_{2g}$ at 268 231
			\pinlabel $\tilde{\alpha}_{2g+1}$ at 250 183
			
			\endlabellist
			\centering
			\includegraphics[width=65ex]{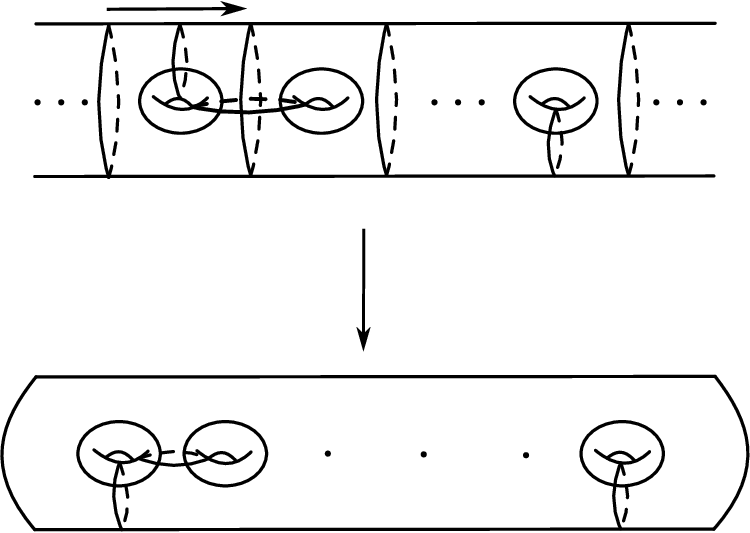}
			\caption{The image of the chain $(\tilde{\alpha}_3, \ldots,\tilde{\alpha}_{2g+1})$ under $q_g$.}
			\label{fig:ladder1}
		\end{figure}

The following is now an immediate consequence of Lemma~\ref{lem:inf_BH_cover}. 

\begin{corollary}
For $g \geq 2$, there exists a short exact sequence 
$$1 \to \langle h^{g-1} \rangle \to \SMod_{q_g}(\L) \to \LMod_{q_g}(S_g) \to 1.$$
\end{corollary}
	
\noindent Let $\displaystyle \UMod(S_g) := \bigcap_{k \geq 2} \LMod_{p_k}(S_g).$ In \cite{ADDR21}, the following  result was derived as an immediate consequence of Theorem~\ref{thm:symp_lift_crit}.
\begin{corollary}[{\cite[Corollary 2.5]{ADDR21}}]
\label{cor:umod_crit}
	Let $f \in \Mod(S_g)$, and let $\Psi(f) = (d_{ij})_{2g \times 2g}$. Then, $f \in \UMod(S_g)$ if and only if $d_{2i} = 0$, for $1\leq i\leq 2g$ and $i \neq 2$, and $d_{22}= \pm 1$.
\end{corollary}
	\noindent Using this criterion, we will now establish the following. 
	
\begin{proposition}
\label{prop:lmod_qg}
For $g \geq 2$, $\UMod(S_g) = \LMod_{q_g}(S_g)$.
\end{proposition}

\begin{proof}
From the definition of $(q_g)_{\#}$ in Equation \ref{eqn:inducedhom}, we have
$$V_g := (q_g)_{\#}(H_1(\mathcal{L},\Z)) = \langle a_1,\ldots,a_g, b_2, \ldots, b_g \rangle.$$ 
It suffices to show that $\Stab_{\Mod(S_g)}(V_g) = \UMod(S_g)$, which is equivalent to showing that 
$$\Psi(\UMod(S_g)) = \Psi(\Stab_{\Mod(S_g)}(V_g)) = \Stab_{\Sp(2g,\Z)}(V_g),$$  since $\ker \Psi = \I(S_g) < \UMod(S_g)$. But it follows immediately from Corollary~\ref{cor:umod_crit} that $\Psi(\UMod(S_g)) <  \Stab_{\Sp(2g,\Z)}(V_g)$. Now our assertion follows from the fact that an $A \in \Stab_{\Sp(2g,\Z)}(V_g)$ has to be of the form described in Corollary~\ref{cor:umod_crit}.

\end{proof}

\noindent Thus, by applying Theorem~\ref{thm1}, Proposition~\ref{prop:lmod_qg}, and \cite[Corollary 3.8]{ADDR21} we have the following corollary. 
\begin{corollary} For $g \geq 3$, 
$$\UMod(S_g) = \langle \S_{g,k}  \cup \{\iota, T_{a_1},T_{a_2},T_{b_2}, \ldots,T_{a_g},T_{b_g},T_{c_1},\ldots,T_{c_{g-1}}\} \rangle.$$
\end{corollary}

The following corollary \cite[Corollary 2]{ADDR21}, which draws inspiration from \cite[Theorem 5.2.2]{TG}, is an immediate consequence of Theorem~\ref{thm:symp_lift_crit}.

\begin{corollary}
\label{cor:lcm_of_covers}
For $g \geq 1$ and $k,\ell \geq 2$, we have
$$\LMod_{p_k}(S_g) \cap \LMod_{p_{\ell}}(S_g) = \LMod_{p_d}(S_g),$$ where
$d = \text{lcm}(k,\ell)$.
\end{corollary}

By Theorem~\ref{thm:norm_series_gen} and Corollary~\ref{cor:lcm_of_covers}, it follows that both $\{\LMod_{p_k}(S_g): k \in \mathbb{N}\}$ and $\{\Mod_{p_k}(S_g,e_1): k \in \mathbb{N}\}$ form directed systems of groups. Furthermore, by Corollary~\ref{cor:umod_crit}, it follows that $\varinjlim \LMod_{p_k}(S_g) = \UMod(S_g)$ and it follows by definition that $\varinjlim \Mod_{p_{k}}(S_g,e_1)$ is the stabilizer in $\Mod(S_g)$ of the vector $e_1 \in H_1(S_g,\Z)$, which we denote by $\Mod(S_g,e_1)$. Thus, by considering the fact that $\iota \in \UMod(S_g) \setminus \Mod(S_g,e_1)$, we obtain the following normal series. 

\begin{corollary}
For $g \geq 2$, there exists a normal series of $\UMod(S_g)$ in $\Mod(S_g)$ given by 
$$1 \lhd \I(S_g) \lhd \Mod(S_g,e_1) \lhd \UMod(S_g),$$ 
where the two distinct cosets of $\UMod(S_g)/\Mod(S_g,e_1)$ are represented by the elements in $\langle \iota \rangle$.
\end{corollary}

\section{Generating $\LMod_{p_2}(S_2)$}
\label{sec:lmod_p2_s2}
Let  $S_{g,n}$ denote the surface of genus $g$ with $n$ punctures (or marked points). The techniques used in Section~\ref{sec:gen_lmod} do not apply to $\LMod_{p_k}(S_2)$ since it is known~\cite{MM} that $\I(S_2)$ is not finitely generated. Let $\delta : S_g \to S_{0,2g+2}$ be the branched cover induced by the action of $\langle \iota \rangle$ on $S_g$. Let $\bar{\delta} : S_{g,2g+2} \to S_{0,2g+2}$ be the covering map obtained from $\delta$ by removing the branch points and their preimages under $\delta$, as shown in Figure~\ref{fig:B5} (for $g = 2$). 
\begin{figure}[H]
		\labellist
		\small
		\pinlabel $\beta_1$ at 514 178
		\pinlabel $\beta_2$ at 543 179
		\pinlabel $\beta_3$ at 575 179
		\pinlabel $\beta_4$ at 606 179
		\pinlabel $\beta_5$ at 638 179
		\pinlabel $\beta_6$ at 667 179
		\pinlabel $\bar{\delta}$ at 385 136
		\pinlabel $\pi$ at 310 145
		\pinlabel $\overline{a_1}$ at 82 103
		\endlabellist
		\centering
		\includegraphics[width=70ex]{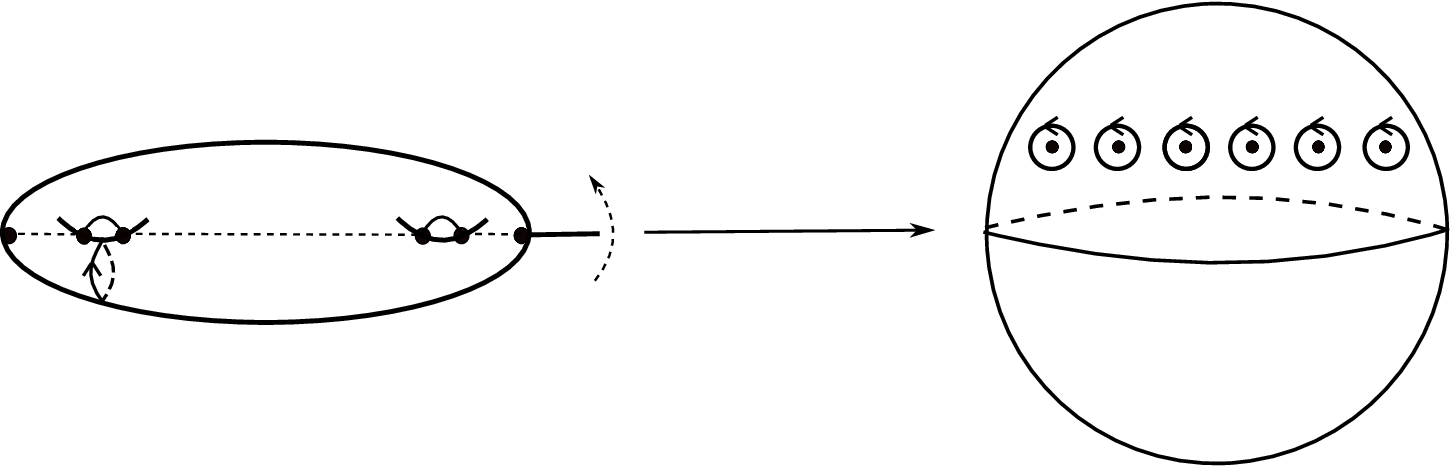}
		\caption{The hyperelliptic cover $\bar{\delta}:S_{2,6} \to S_{0,6}$.}
		\label{fig:B5}
	\end{figure}
\noindent We label the marked (branched) points of $S_{0,2g+2}$ with the numbers $1,2,\ldots,2g+2$. Let $\bar{\delta}_{\#}: H_1(S_{g,2g+2}, \Z_2) \to H_1(S_{0,2g+2},\Z_2)$ be the induced map on homology. For $1 \leq i \leq 2g+2$, let $\beta_i \in H_1(S_{0,2g+2},\Z_2)$ be represented by the curve (oriented counterclockwise) enclosing the marked point $i$ in $S_{0,2g+2}$. We will employ an alternative approach to derive a finite generating set for $\LMod_{p_2}(S_2)$ using the following classical result from Birman-Hilden theory~\cite{BH3}. 

\begin{theorem}
\label{thm:BH}
For $g \geq 2$, let $\iota \in \Mod(S_g)$ denote the hyperelliptic involution. Then following sequence is exact: 
$$1 \to \langle \iota \rangle \to C_{\Mod(S_g)}(\iota) \xrightarrow{\hat{\delta}} \Mod(S_{0,2g+2}) \to 1,$$ where $\hat{\delta}$ is induced by the hyerelliptic cover $\delta : S_g \to S_{0,2g+2}$. In particular, when $g = 2$, this sequence takes the form 
$$1 \to \langle \iota \rangle \to \Mod(S_2) \xrightarrow{\hat{\delta}} \Mod(S_{0,6}) \to 1.$$
\end{theorem}

\noindent Note that for $1 \leq i \leq 2g+1$, we have $\hat{\delta}(T_{\alpha_i}) = \sigma_i$, where $\sigma_i$ is the (left-handed) half-twist about the arc connecting the point $i$ with $i+1$ and the $\alpha_i$ are the chain curves shown in Figure~\ref{fig:ladder1}. 

\begin{remark}
\label{rem:hyper_ellip}
Let $I_{2g} \in \Sp(2g,\Z_k)$ be the identity matrix. Since $\Psi_k(\iota) = -I_{2g}$, it follows by definition of $\Mod_{p_k}(S_g, e_1)$ that $\iota \in \Mod_{p_k}(S_g, e_1)$ if and only if $k=2$. Hence, for $k \geq 3$, we have $C_{\Mod(S_g)}(\iota) \cap \Mod_{p_k}(S_g, e_1) \xhookrightarrow{\hat{\delta}} \Mod(S_{0,2g+2}).$ 
\end{remark} 

 \noindent Observing from Figure~\ref{fig:B5} that  $\bar{\delta}_{\#}(\overline{a_1}) = \beta_1 + \beta_2$, we make the following claim. 

\begin{proposition}
\label{lem:stab_vector}
For $g \geq 2$, we have: 
\begin{enumerate}[(i)]
\item $\hat{\delta}(\LMod_{p_2}(S_g) \cap C_{\Mod(S_g)}(\iota)) = \Stab_{\Mod(S_{0,2g+2})}(\beta_1+\beta_2),$ and
\item $\LMod_{\delta\circ p_2}(S_{0,2g+2}) = \hat{\delta}(\LMod_{p_2}(S_g) \cap C_{\Mod(S_g)}(\iota))$.
\end{enumerate}
\end{proposition}

\begin{proof}
For $ 1\leq i \leq 2g+2$, let $x_i \in H_1(S_{g,2g+2}, \Z_2)$ be represented by a loop (oriented counterclockwise) going around the lift of the point $i \in S_{0,2g+2}$ under $\bar{\delta}$ so that 
$\bar{\delta}_{\#}(x_i) = 2\beta_i = 0$. Now consider homomorphism $j_{\#}:H_1(S_{g,2g+2},\mathbb{Z}_2)\rightarrow H_1(S_{g},\mathbb{Z}_2)$ induced by the inclusion map $j : S_{g,2g+2} \to S_g$. Then it is apparent that 
$$\ker(j_{\#})= \ker(\bar{\delta}_{\#}) = \langle\{ x_i : 1 \leq i \leq 2g+2\} \rangle.$$

Let $\phi \in \Mod_{p_2}(S_g,e_1) \cap C_{\Mod(S_g)}(\iota)$, and let $f$ be a fiber-preserving representative of $\phi$. Consider the surjective homomorphism $Forget: \Mod(S_{g,2g+2}) \to \Mod(S_g)$ that forgets the $2g+2$ marked points, induced by the inclusion $j: S_{g,2g+2} \to S_g$, and let $Forget(\bar{\phi})=\phi$. Choose a representative $\bar{f}$ for $\bar{\phi}$ such that $j\circ \bar{f}=f\circ j$. 

Since $f_{\#}(a_1)=a_1$, we have $\bar{f}_{\#}(\overline{a_1})=\overline{a_1}+x$, where  $x\in \ker(j_{\#})$ and $j_{\#}(\overline{a_1}) = a_1$ (see Figure~\ref{fig:B5}). Thus, $\bar{\delta} \circ \bar{f}=\hat{\delta}(f)\circ \bar{\delta}$, which would imply that
$$\hat{\delta}(f)_{\#}\circ \bar{\delta}_{\#}(\overline{a_1})=\bar{\delta}_{\#}\circ \bar{f}_{\#}(\overline{a_1})=\bar{\delta}_{\#}(\overline{a_1}+x),$$
and so it follows that $$(\hat{\delta}(f))_{\#}(\beta_1+\beta_2)=\beta_1+\beta_2.$$ 
The converse is just a matter of reversing this argument. This concludes the argument for (i).  

For (ii), it suffices to show that $\LMod_{\delta\circ p_2}(S_{0,2g+2}) \subset \hat{\delta}(\LMod_{p_2}(S_g) \cap C_{\Mod(S_g)}(\iota))$, for the other inclusion follows easily by definition. Let $f \in \LMod_{\delta\circ p_2}(S_{0,2g+2})$ lift to an $\tilde{f} \in N_{\Mod(S_{g_2})}(\langle \iota, R \rangle)$, where $R$ is represented by the free rotation $r$ of $S_{g_2}$ by $\pi$ inducing the cover $p_2$. Then $\tilde{f}^{-1} \circ R \circ \tilde{f} \in \langle \iota, R \rangle$, and so $\tilde{f}^{-1} \circ R \circ \tilde{f}$ is represented by a $\varphi \in \Homeo^{+}(S_{g_2})$ that has no fixed points. Thus, $\varphi = r$, and so we have $\tilde{f} \in C_{\Mod(S_{g_2})}(R)$, from which our assertion follows.
\end{proof}

\begin{theorem}
\label{thm:symm_mod_gen}
For $g \geq 2$, 
$$\LMod_{p_2}(S_g) \cap C_{\Mod(S_g)}(\iota) = \langle \iota, T_{a_1},T_{a_g}, T_{b_1}^2, T_{b_2}. \ldots T_{b_g}, T_{c_1}, \ldots, T_{c_{g-1}}\rangle.$$
\end{theorem}
\begin{proof}
Let $\delta : S_g \to S_{0,2g+2}$ be the branched cover induced by the action of $\langle \iota \rangle$ on $S_g$, and let $$\bar{\delta}_{\#}: H_1(S_{g,2g+2}, \Z_2) \to H_1(S_{0,2g+2},\Z_2)$$ be the induced map on homology.  For $1 \leq i \leq 2g+2$, let $\beta_i \in H_1(S_{0,2g+2},\Z_2)$ be represented by the curve (oriented counterclockwise) enclosing the marked point $i$ in $S_{0,2g+2}$. Then $\bar{\delta}_{\#}(\overline{a_1}) = \beta_1 + \beta_2$, and by Proposition~\ref{lem:stab_vector}, we have
$$\hat{\delta}(\Mod_{p_2}(S_g, e_1) \cap C_{\Mod(S_g)}(\iota)) = \Stab_{\Mod(S_{0,2g+2})}(\beta_1+\beta_2).$$ 

Consider the natural epimorphism $\rho: \Mod(S_{0,2g+2}) \to \Sigma_{2g+2}$, where $\Sigma_{2g+2}$ denotes the group of permutations of the points $\{1,\ldots,2g+2\}$. It is apparent that 
$$\rho(\Stab_{\Mod(S_{0,2g+2})}(\beta_1+\beta_2)) = \langle \{(1,2)\} \cup \{(3,4),(4,5),\dots, (2g+1,2g+2)\} \rangle,$$
where $(i,j) \in \Sigma_{2g+2}$ denotes the transposition of the points $i$ and $j$. Consequently, we have
$$\Stab_{\Mod(S_{0,2g+2})}(\beta_1+\beta_2) = \langle \ker  \rho  \cup \{\sigma_1\} \cup \{\sigma_3, \sigma_4, \ldots, \sigma_{2g+1}\}  \rangle.$$
\noindent Moreover, it is well known that 
$$\ker \rho = \mathrm{PMod}(S_{0,2g+2}) = \langle \{\tau_{ij} : 1\leq i < j\leq 2g+2\} \rangle,$$ where
$\tau_{ij} = (\sigma_{j-1}\dots \sigma_{i+1})\sigma_i^2(\sigma_{j-1}\dots \sigma_{i+1})^{-1}$. So, we have that
$$\tau_{ij} \in \langle \{\sigma_1,\sigma_2^2\} \cup \{\sigma_3,\sigma_4, \ldots, \sigma_{2g+1}\}\rangle,$$ whenever $\tau_{ij}$ does not contain the subword $\tau_{13} = \sigma_2\sigma_1^2 \sigma_2^{-1}$. But by repeated application of the braid relation, we can see that:
$$ \sigma_2 \sigma_1 \sigma_1 \sigma_2^{-1}=\sigma_{1}^{-1}\sigma_2 \sigma_1 \sigma_2 \sigma_1 \sigma_2^{-1} =\sigma_1^{-1}\sigma_{2} \sigma_{2} \sigma_{1} \sigma_{2} \sigma_{2}^{-1} =\sigma_{1}^{-1}\sigma_2^{2}\sigma_{1},$$ and so it follows that $\tau_{13} \in \langle \{\sigma_1,\sigma_2^2\} \cup \{\sigma_3,\sigma_4, \ldots, \sigma_{2g+1}\}\rangle.$

Since $\hat{\delta}(T_{\alpha_i}) = \sigma_i$ and $\S_2'' = \{1\}$, our assertion follows from Remark~\ref{rem:hyper_ellip}.
\end{proof}

\begin{corollary}
$\LMod_{p_2}(S_2) = \langle T_{a_1}, T_{b_1}^2, T_{c_1}, T_{a_2}, T_{b_2} \rangle$.
\end{corollary}

\begin{proof}
In view of Theorems~\ref{thm:BH} and~\ref{thm:symm_mod_gen}, it suffices to show that $\iota \in  \langle T_{a_1}, T_{b_1}^2, T_{c_1}, T_{a_2}, T_{b_2} \rangle$. It is known~\cite[Lemma 3.1]{NN} that $\iota \in \Mod(S_2)$ has a representation of the form $(T_{a_1}T_{b_1})^{3}(T_{a_2}T_{b_2})^{-3}$. By the braid relation, we see that
\begin{eqnarray*}
\iota & = & (T_{a_1}T_{b_1}(T_{a_1}T_{b_1}T_{a_1})T_{b_1})(T_{a_2}T_{b_2})^{-3} \\
        &  = & (T_{a_1}T_{b_1}(T_{b_1}T_{a_1}T_{b_1})T_{b_1})(T_{a_2}T_{b_2})^{-3} \\
        & = & (T_{a_1}T_{b_1}^{2})^{2}(T_{a_2}T_{b_2})^{-3} \in \langle T_{a_1}, T_{b_1}^2, T_{c_1}, T_{a_2}, T_{b_2} \rangle,
 \end{eqnarray*}       
from which our assertion follows.
\end{proof}

\begin{corollary}
For $g \geq 2$, we have $$\LMod_{\delta \circ p_2}(S_{0,2g+2}) =  \langle \hat{\delta} (\{\iota, T_{a_1},T_{a_g}, T_{b_1}^2, T_{b_2}. \ldots T_{b_g}, T_{c_1}, \ldots, T_{c_{g-1}}\}) \rangle.$$
\end{corollary}

\section*{Acknowledgments}
The authors would like to thank Prof. Dan Margalit for motivating the results in this paper and also providing some of the key ideas behind the results in Section~\ref{sec:lmod_p2_s2}. The second author was supported by the NBHM Postdoctoral Fellowship and the fourth author is supported by the SERB MATRICS grant of the Govt. of India.

\bibliographystyle{plain}
\bibliography{LMod_Gen}
	
\end{document}